\documentclass[a4paper, 10pt]{amsart}

\usepackage{scrextend}

\makeatletter
\newcommand\footnoteref[1]{\protected@xdef\@thefnmark{\ref{#1}}\@footnotemark}
\makeatother

\newcommand{\lmfdbec}[3]{\href{https://www.lmfdb.org/EllipticCurve/Q/#1/#2/#3}{#1.#2#3}}
\newcommand{\lmfdbeciso}[2]{\href{https://www.lmfdb.org/EllipticCurve/Q/#1/#2}{#1.#2}}
\newcommand{\lmfdbecnf}[4]{\href{https://www.lmfdb.org/EllipticCurve/#1/#2/#3/#4}{#1-#2-#3#4}}
\newcommand{\lmfdbecnfiso}[3]{\href{https://www.lmfdb.org/EllipticCurve/#1/#2/#3}{#1-#2-#3}}

\usepackage{amsmath,amsfonts,amssymb,amsthm,url,verbatim}
\usepackage[dvips]{graphicx}
\usepackage[croatian,english]{babel}
\usepackage{amsfonts}
\usepackage[utf8]{inputenc}
\usepackage[T1]{fontenc}
\usepackage{amsmath}
\usepackage{amssymb}
\usepackage{latexsym}
\usepackage[usenames,dvipsnames]{xcolor}
\usepackage{soul}

\usepackage{booktabs}
\usepackage{longtable}
\usepackage{multirow}

\usepackage[backref]{hyperref}
\usepackage{cleveref}

\usepackage{ifthen}
\usepackage{graphpap}

\renewcommand{\a}{\mathfrak{a}}
\newcommand{\p}{\mathfrak{p}}
\newcommand{\q}{\mathfrak{q}}

\newcommand{\Q}{\mathbb{Q}}
\newcommand{\C}{\mathbb{C}}
\newcommand{\Z}{\mathbb{Z}}

\newcommand{\F}{\mathbb{F}}

\newcommand{\OO}{\mathcal{O}}
\newcommand{\Qbar}{{\overline{\mathbb Q}}}
\newcommand{\GQ}{G_{{\mathbb Q}}}
\def\diam#1{\langle#1\rangle}

\def\Sage{{\tt SageMath}}
\def\Magma{{\tt Magma}}
\def\True{{\tt True}}
\def\False{{\tt False}}

\DeclareMathOperator{\tors}{tors}
\DeclareMathOperator{\cond}{cond}

\usepackage{amsmath,amsfonts,amssymb}
\usepackage{url}
\usepackage{comment}
\usepackage[all]{xy}
\usepackage{xcolor}

\newtheorem{thm}{Theorem}[section]
\newtheorem{cor}[thm]{Corollary}
\newtheorem{lem}[thm]{Lemma}
\newtheorem{prop}[thm]{Proposition}
\theoremstyle{remark}
\newtheorem{defn}{Definition}
\newtheorem*{rmk}{Remark}

\numberwithin{equation}{section}

\def\Z{{\mathbb Z}}
\def\Q{{\mathbb Q}}
\def\Qbar{\overline{\mathbb Q}}

\def\C{{\mathbb C}}

\def\P{{\mathbb P}}
\def\F{{\mathbb F}}
\def\OO{{\mathcal O}}
\def\CC{{\mathcal C}}
\def\JJ{{\mathcal J}}

\def\PP{{\mathcal P}}
\def\QQ{{\mathcal Q}}
\def\eps{\varepsilon}
\def\rhobar{\overline{\rho}}

\DeclareMathOperator{\End}{End}
\DeclareMathOperator{\Aut}{Aut}

\DeclareMathOperator{\disc}{disc}

\DeclareMathOperator{\ord}{ord}
\DeclareMathOperator{\ch}{char}

\DeclareMathOperator{\Frob}{Frob}
\DeclareMathOperator{\Gal}{Gal}
\DeclareMathOperator{\GL}{GL}

\DeclareMathOperator{\PGL}{PGL}
\DeclareMathOperator{\PSL}{PSL}
\DeclareMathOperator{\AGL}{AGL}

\def\<#1>{\left<#1\right>}

\newcommand{\jbar}{G(j)}

\begin{document}
\title{$\Q$-curves over odd degree number fields}
\author{J. E. Cremona}
\address{Mathematics Institute, University of Warwick, Coventry CV4 7AL, UK}
\email{j.e.cremona@warwick.ac.uk}
\author{Filip Najman}
\address{Department of Mathematics, Faculty of Science University of Zagreb, Bijenička cesta 30, 10000 Zagreb, Croatia}
\email{fnajman@math.hr}
\date{\today}

\thanks{Cremona was supported by the Heilbronn Institute for
  Mathematical Research.}

\thanks{Najman was supported by the QuantiXLie Centre of Excellence, a
  project co-financed by the Croatian Government and European Union
  through the European Regional Development Fund - the Competitiveness
  and Cohesion Operational Programme (Grant KK.01.1.1.01.0004) and by
  the Croatian Science Foundation under the project
  no. IP-2018-01-1313.}

\begin{abstract}
By reformulating and extending results of Elkies, we prove some
results on $\Q$-curves over number fields of odd degree. We show that,
over such fields, the only prime isogeny degrees~$\ell$ that a $\Q$-curve without CM may have are those degrees that are already
possible over~$\Q$ itself (in particular, $\ell\le37$), and we show
the existence of a bound on the degrees of cyclic isogenies between
$\Q$-curves depending only on the degree of the field.  We also prove
that the only possible torsion groups of $\Q$-curves over number
fields of degree not divisible by a prime $\ell\leq 7$ are the $15$
groups that appear as torsion groups of elliptic curves over $\Q$.
Complementing these theoretical results, we give an algorithm for
establishing whether any given elliptic curve $E$ is a $\Q$-curve,
that involves working only over $\Q(j(E))$.
\end{abstract}

\maketitle

\section{Introduction}
\label{sec:intro}

In the study of elliptic curves over number fields, $\Q$-curves are of
special interest. An elliptic curve is called a \emph{$\Q$-curve} if
it is isogenous to all of its $\Gal( \Qbar /\Q)$-conjugates.
Throughout the paper, ``isogenous'', when said without specifying the
field, will always mean ``isogenous over $\Qbar$''.

This property is obviously satisfied by all elliptic curves defined
over $\Q$, and more generally all elliptic curves with rational
$j$-invariants; also, all curves with complex multiplication (CM) are
$\Q$-curves.  The property of being a $\Q$-curve is preserved under isogeny, and
$\Q$-curves not isogenous to an elliptic curve with rational $j$-invariant are called \emph{strict}
$\Q$-curves. Thus, $\Q$-curves can be thought of as generalizations of
elliptic curves defined over $\Q$ (or, more generally, elliptic curves
with rational $j$-invariants). Moreover, Ribet proved in \cite{ribet}
(assuming Serre's conjecture, which has since been proved \cite{khw,
  khw2}) that $\Q$-curves are exactly the elliptic curves over number
fields that are modular, in the sense of being quotients of $J_1(N)$
for some $N$.

As can be seen from \cite{elkies}, and as will be later explained in
more detail, $\Q$-curves in a sense arise most naturally over
number fields of degree~$2^n$. In particular, the first place one
looks for $\Q$-curves that are not just elliptic curves defined over
$\Q$, are among those defined over quadratic fields.  Already, over
quadratic fields, there is a plethora of results showing that
$\Q$-curves have certain special properties, of which we list some
examples. Le Fourn \cite{lf} showed that for every strict $\Q$-curve
$E$ over a fixed imaginary quadratic field $K$, there exists a uniform
bound $C_K$ such that for $\ell>C_K$, the mod $\ell$ representation
attached to $E$ is surjective. Bruin and Najman \cite{BN} and Box
\cite{box} showed that for each $N$ such that the modular curve $X_0(N)$
is hyperelliptic, for all but finitely many explicitly listed
exceptional elliptic curves, an elliptic curve over a quadratic field with an
$N$-isogeny is a $\Q$-curve.  Bosman, Bruin, Dujella and Najman
\cite{bbdn} showed that all elliptic curves with $\Z/13\Z$, $\Z/16\Z$
and $\Z/18\Z$-torsion over quadratic fields are again $\Q$-curves. Le
Fourn and Najman \cite{lf} determined all the possible torsion groups
of $\Q$-curves over quadratic fields.

The main purpose of this paper is to expand on the existing theory of
$\Q$-curves and study their properties, especially over odd degree
number fields.  To this end, we reformulated and expanded on the work
of Elkies \cite{elkies} concerning $\Q$-curves and their
generalizations.  A summary of the definitions and results about
properties of $\Q$-curves obtained by this reformulation can be found
in \Cref{sec:Q-curve-summary}, with detailed proofs in the Appendix. A
key result here (\Cref{prop:isog-to-central}) is that every non-CM
$\Q$-curve~$E$ defined over a number field~$K$ is isogenous over~$K$
to a \emph{central} $\Q$-curve defined over a polyquadratic subfield
of~$K$.  This result is established through the concept of the
\emph{core} of the isogeny class of a $\Q$-curve, that is defined
over a polyquadratic field.  We thereby obtain the main new results in
this section, \Cref{odd_degree} and \Cref{no_quadratic_subfields},
which state that $\Q$-curves defined over number fields of odd degree,
or more generally fields without quadratic subfields, are always
isogenous to elliptic curves defined over $\Q$.

In \Cref{sec:isogenies} we study the possible degrees of isogenies of $\Q$-curves over odd degree number fields. \Cref{odd_degree} allows us to bound both the degree of the isogeny and the size of the torsion group of a $\Q$-curve over an odd degree number field. These results fit into the long-standing program, going back to Levi and Ogg and whose most famous results include Mazur's torsion and isogeny theorems \cite{mazur,mazur2}, of describing the possible torsion groups and isogeny structures of elliptic curves over number fields. Our results are reminiscent of Merel's uniform boundedness theorem \cite{merel}, but we will obtain absolute bounds, not depending even  on the degree of the number field, on degrees of isogenies of non-CM $\Q$-curves over odd degree number fields.

In \Cref{sec:isogenies} we obtain the following results.

\begin{thm}\label{tm:bounds}
  Let $E$ be a $\Q$-curve without complex multiplication defined over an odd degree number field $K$. Then
  \begin{itemize}
    \item[a)] If $E$ has a $K$-rational isogeny of prime degree $\ell$, then $\ell \in \{2,3,5,7,11,13,17,37\}$.
    \item[b)] If $d=[K:\Q]$ is not divisible by any prime $\ell \in \{2,3,5,7,11,13,17,37\}$, and $E$ has a cyclic isogeny of degree $n$, then $n\leq 37$.
  \end{itemize}
\end{thm}

Clearly, if $E$
is an elliptic curve over $\Q$ and $\ell$ is any prime, then for a
suitable extension $K/\Q$ (generically of degree $\ell+1$), the
base-change of $E$ from $\Q$ to $K$ will acquire an $\ell$-isogeny.
The point of \Cref{tm:bounds} is that except for the eight primes~$\ell$
listed, $K$ will always have even degree.

We show that even if we include CM curves, over odd degree number fields there exists a
uniform bound, depending only on the degree of the number field, on
the degree of isogenies of all $\Q$-curves.

\begin{thm}\label{bound:isogenies}
  For every odd positive integer~$d$, there exists a bound $C_d$
  depending only on~$d$ such that all cyclic isogenies of all
  $\Q$-curves over all number fields of degree $d$ are of degree at
  most~$C_d$.
\end{thm}

We expect that \Cref{bound:isogenies} should also hold for even
degrees $d$, once CM curves have been excluded, but new methods would
be required to prove such a general uniformity result.

By \Cref{odd_degree}, for elliptic curves without CM, our problem is
equivalent to studying isogenies of elliptic curves $E$ with $j(E)\in
\Q$ over odd degree number fields. Isogenies of elliptic curves $E$
with $j(E)\in \Q$ without CM have been studied by Najman in
\cite{naj}, but over general number fields. As we will see, if one
restricts to odd degree number fields, we get much sharper results. In
\cite{propp}, Propp also studied the degrees of extensions over which
an elliptic curve with $j(E) \in \Q$ have certain kinds of Galois
images.

In \Cref{sec:torsion}, we study the possible torsion groups of
elliptic curves over odd degree number fields. While studying torsion
groups of $\Q$-curves over number fields of prime degree, we will not
need to restrict to elliptic curves with CM. Our main result in
\Cref{sec:torsion} is the following theorem, which is a generalization
of \cite[Theorem 7.2. (i)]{gn}.

\begin{thm}\label{bound:torsion}
  Let $p$ be a prime $>7$, let $K$ be a number field of degree $p$ and
  $E/K$ a $\Q$-curve. Then $E(K)_{\tors}$ is one of the groups from
  Mazur's theorem (listed in (\ref{eqn:MazurList})), i.e., a torsion
  group of an elliptic curve over $\Q$.
\end{thm}

Note that this result does not hold for $p=2,3,5$.  For example, the
$\Q$-curves with LMFDB labels \lmfdbecnf{2.2.17.1}{100.1}{e}{2} over
$\Q(\sqrt{17})$ and \lmfdbecnf{3.3.49.1}{27.1}{a}{2} over the cubic
subfield of $\Q(\zeta_7)$ both have torsion order~$13$, and the
$\Q$-curve with label \lmfdbecnf{5.5.14641.1}{121.1}{a}{1} over the
quintic subfield of $\Q(\zeta_{11})$ has torsion order~$11$.  We do
not know of counterexamples in degree~$7$, and it is possible that our
methods might be extended to include that case.

Since elliptic curves with rational $j$-invariants are $\Q$-curves,
these theorems apply to all such curves.

In \Cref{sec:algorithm} we address the question of how to test a given
elliptic curve~$E$ defined over a number field~$K$ for the property of
being a~$\Q$-curve.  Using the results of \Cref{sec:Q-curve-summary}
proved in the Appendix, we are able to give an algorithm that solves
this problem without needing to extend the base field (for example, to
the Galois closure of~$K$).  We assume that we can detect CM, and
can compute the complete $K$-isogeny class of any elliptic curve defined
over a number field~$K$, both of which are already implemented
in~\Sage \cite{sage}, the former also in \Magma \cite{magma}.  One
special case we establish (see \Cref{cor:odd_degree}) is that, if $E$
does not have CM and $K$ has no quadratic subfields, then $E$ is a
$\Q$-curve if and only if it is isogenous over~$K$ to a curve with
rational $j$-invariant.  We have implemented this algorithm in \Sage\,
and used it to establish which of the curves in the LMFDB database
(see~\cite{lmfdb}), which (as of August 2020) are defined over fields
of degree at most~$6$, are $\Q$-curves.  Our \Sage~code is available
at \cite{ecnf}.

All examples of elliptic curves given in the text are identified with
their LMFDB labels and may be found in the LMFDB database~\cite{lmfdb}.

\subsection*{Acknowledgments}
We thank Samuel Le Fourn for helpful conversations and suggestions,
Andrew Sutherland for the references~\cite{BLS}
and~\cite{SutherlandDatabase}, Abbey Bourdon for several useful
remarks on a previous version of the paper and the referees for their
careful reading and helpful suggestions for improving the exposition;
we have added many details to several proofs at the referee's request.

\section{Properties of $\Q$-curves}
\label{sec:Q-curve-summary}
We recall here the definition of a $\Q$-curve and various related
concepts.  Proofs of all the properties stated in this section are
given in the Appendix.

Let~$\Qbar$ be the field of algebraic numbers, and
$\GQ=\Gal(\Qbar/\Q)$.  A \emph{$\Q$-curve} is an elliptic curve~$E$
defined over $\Qbar$ such that $E$ is isogenous (over $\Qbar$) to all
its Galois conjugates.  A \emph{$\Q$-number} is an algebraic
number~$j$ that is the $j$-invariant of a $\Q$-curve. If~$j$ is a
$\Q$-number then so are all its Galois conjugates (see
\Cref{prop:Qclass}).  All CM curves are $\Q$-curves; however, here we
will mainly be interested in non-CM $\Q$-curves.

Two algebraic numbers~$j_1,j_2$ are \emph{isogenous} if there are two
isogenous elliptic curves~$E_i$ defined over~$\Qbar$ with
$j(E_i)=j_i$, in which case every pair of elliptic curves with these
$j$-invariants are isogenous over~$\Qbar$.  Isogeny is an equivalence
relation on~$\Qbar$.  If $j_1$ and~$j_2$ are isogenous and not CM,
then there is a unique positive integer~$d$ that is the degree of a
\emph{cyclic} isogeny $E_1\to E_2$, where again~$j(E_i)=j_i$, denoted
$d(j_1,j_2)$ (see \Cref{lem:cyclic-deg}).

A \emph{$\Q$-class} is an isogeny class~$\QQ\subset\Qbar$ consisting
of $\Q$-numbers.  The \emph{isogeny degree} of a $\Q$-number~$j$ is the least
common multiple of the degrees $d(j,g(j))$ for $g\in\GQ$.  A
$\Q$-number is~$\emph{central}$ if it has square-free isogeny degree, in which
case its Galois conjugacy class is called a \emph{central class}.  The
existence of a central class in every~$\Q$-class is established in
\Cref{thm:core-exists}, and in \Cref{thm:central-class-properties} and
\Cref{prop1:N=1} we prove the other assertions of the following
theorem:
\begin{thm}
\label{thm:central-class-properties-intro}
  Let $\QQ$ be a non-CM $\Q$-class in~$\Qbar$. Then $\QQ$ contains at least
  one central conjugacy class, and each central class~$C$ in~$\QQ$ satisfies the
  following properties:
  \begin{enumerate}
  \item $|C|=2^\rho$ for some~$\rho\ge0$;
  \item $\Q(C)$ is a polyquadratic field with
    $\Gal(\Q(C)/\Q)\cong(\Z/2\Z)^\rho$;
  \item the isogeny degree~$N$ of one (and hence all) $j\in C$ is a
    product of $r$ distinct prime factors, where $r\ge\rho$ and
    $r=0\iff\rho=0$.
  \end{enumerate}
  The quantities~$N$, $r$ and~$\rho$, and the central field~$\Q(C)$, are the
  same for each central class in~$\QQ$.
\end{thm}

We denote the integers~$N$, $r$, and~$\rho$ attached to any central
class~$C$ in the $\Q$-class~$\QQ$ by $N(\QQ)$, $r(\QQ)$ and
$\rho(\QQ)$; similarly, the central polyquadratic field $\Q(C)$ of
degree~$2^\rho$ is denoted $L_\QQ$.  We call~$N(\QQ)$ the \emph{level}
of the $\Q$-class~$\QQ$.  The set of degrees of the isogenies between
elements of each central class~$C$ in~$\QQ$ has size~$2^{\rho(\QQ)}$,
and forms a subgroup under multiplication modulo squares of the group
of all~$2^{r(\QQ)}$ divisors of~$N(\QQ)$.

By applying Atkin--Lehner involutions (see \Cref{subsec:AL}) to the
isogenies between elements of a central class~$C$ we obtain a
\emph{core} of the $\QQ$-class.  This has cardinality~$2^{r(\QQ)}$ and
consists of $2^{r(\QQ)-\rho(\QQ)}$ disjoint central classes.  The
degrees of the isogenies between elements of the core are all
$2^{r(\QQ)}$ divisors of the level.

Note that when we refer to the degree of algebraic numbers~$j$ in the
remainder of this section, we mean the usual degree of the
extension~$\Q(j)/\Q$, and not its isogeny degree as defined above for
a $\Q$-number~$j$.

\begin{prop}
\label{prop:rho-divides-deg-j}
Let $j$ be a $\Q$-number in the non-CM $\Q$-class~$\QQ$.  Then
$L_{\QQ}\subseteq\Q(j)$, and the degree of the field~$\Q(j)$ is
divisible by~$2^{\rho(\QQ)}$.
\end{prop}
\begin{proof}
The inclusion $L_{\QQ}\subseteq\Q(j)$ is part of
\Cref{thm:central-class-properties}, and~$L_{\QQ}$ has
degree~$2^{\rho(\QQ)}$.
\end{proof}

\begin{prop}\label{prop2.3}
  For a non-CM $\Q$-class~$\QQ$, the following are equivalent:
  \begin{enumerate}
    \item $r(\QQ)=0$;
    \item $\rho(\QQ)=0$;
    \item $N(\QQ)=1$;
    \item $L_\QQ=\Q$;
      \item $\QQ\cap\Q\neq\emptyset$.
  \end{enumerate}
\end{prop}
\begin{proof}
The first four are equivalent by
\Cref{thm:central-class-properties-intro}, and the last by
\Cref{prop:rho-divides-deg-j}.
\end{proof}
A $\Q$-class is called \emph{rational} when it satisfies one (and hence all) of the conditions of \Cref{prop2.3}.

\begin{thm}[The odd degree theorem]
\label{odd_degree}
If the non-CM $\Q$-class~$\QQ$ contains an element $j$ such that
$\Q(j)$ has odd degree, then $\QQ$ is rational.
\end{thm}
\begin{proof}
$[\Q(j):\Q]$ is an odd multiple of~$2^{\rho(\QQ)}$ by
  \Cref{prop:rho-divides-deg-j}, so $\rho(\QQ)=0$.
\end{proof}

\begin{rmk}
The assumption in \Cref{odd_degree} that $\QQ$ does not have CM is
necessary.  For example, let $p$ be a prime such that $p\equiv 3 \pmod
4$ and $p>163$, let $\OO=\Z\left[\frac{1+\sqrt{-p}}{2}\right]$ be the
order of discriminant~$-p$, and let $j=j((1+\sqrt{-p})/2)$.  Since
$\OO$ has prime discriminant, its class number~$h$ is odd (by genus
theory), and~$h>1$ since $p>163$.  Hence (by standard facts from Class
Field Theory), $j$ is an algebraic integer of odd degree~$h>1$.  Every
element of~$\QQ(j)$ is the $j$-invariant of an elliptic curve with CM
by an order in~$\Q(\sqrt{-p})$, so has degree a multiple of~$h$. Hence
$\QQ(j)\cap\Q=\emptyset$.
\end{rmk}

More generally:
\begin{thm}[The no-quadratic-subfields theorem]
\label{no_quadratic_subfields}
  If the non-CM $\Q$-class~$\QQ$ contains
  an element $j$ such that $\Q(j)$ has no quadratic subfields, then
  $\QQ$ is rational.
\end{thm}
\begin{proof}
If $\rho(\QQ)\ge1$ then $L_{\QQ}$ has a quadratic subfield, hence so
does~$\Q(j)$, by \Cref{prop:rho-divides-deg-j}.
\end{proof}

A non-CM $\Q$-curve~$E$ is \emph{central} if its $j$-invariant is
central; that is, if the least common multiple of the degrees of the
cyclic isogenies between~$E$ and its Galois conjugates is squarefree.
By \Cref{thm:central-class-properties-intro}, the $j$-invariant of a
central $\Q$-curve is always of degree a power of~$2$, and the
field~$\Q(j)$ is polyquadratic.  In the simplest case of a rational
$\Q$-class, the central $j$-invariants are actually rational, so the
corresponding elliptic curves are quadratic twists of curves defined
over~$\Q$.

In the Appendix we show that, for every $\Q$-curve~$E$ defined over a
number field~$K$, the $K$-isogeny class of~$E$ itself contains a
central $\Q$-curve.  This will enable us
to show whether or not an elliptic curve is a $\Q$-curve without
needing to extend the base field; this is important algorithmically. The following theorem is proved in the Appendix in \Cref{cor:isog-degree}.
\begin{thm}
  \label{prop:isog-to-central}
  Let $K$ be a number field and let $E$ be a non-CM $\Q$-curve defined
  over~$K$.  Then there exists a central $\Q$-curve~$E_0$ with an
  isogeny~$\phi\colon E\to E_0$, where both~$E_0$ and~$\phi$ are also
  defined over~$K$.
\end{thm}

The following are immediate consequences.

\begin{thm}
  \label{cor:odd_degree}
  Let $E$ be a non-CM $\Q$-curve defined over a number field~$K$.  If
either~$\Q(j(E))$ has odd degree, or more generally if $\Q(j(E))$ has
no quadratic subfields, then $E$ is isogenous over~$K$ to an elliptic
curve with rational $j$-invariant.
\end{thm}

\begin{cor}
  \label{cor:no_quadratic_subfields}
  Let $E$ be a non-CM $\Q$-curve defined over a number field~$K$.  If
$\Q(j(E))$ has degree~$4$, with Galois closure isomorphic to either
$S_4$ or~$A_4$, then $E$ is isogenous over~$K$ to an elliptic curve
with rational $j$-invariant.
\end{cor}
\begin{proof}
The one-point stabilisers in both~$S_4$ and $A_4$ have index~$4$ but
are maximal, so $\Q(j(E))$ has no quadratic subfields.
\end{proof}

\section{Isogenies of $\Q$-curves over odd degree number fields}
\label{sec:isogenies}

Let $K$ be a number field and~$G_K=\Gal (\overline K/K)$ its absolute Galois
group.  Let $E/K$ be an elliptic curve, $P\in E[\ell]$ be a point of
order $\ell$ and $C=\diam{P}$ be the subgroup of $E$ generated by
$P$. We define $K(P)$ to be the field obtained by adjoining the
coordinates of $P$ to $K$ and $K(C)$ to be smallest extension of $K$
such that the $\ell$-isogeny $\phi$ with kernel $C$ is defined over
$K(C)$, or in other words, the smallest number field such that
$\Gal(\overline{K(C)}/K(C))$ acts on $C$. Now $K(C)$ and $K(P)$ lie in
a tower of extensions of number fields $K(E[\ell])/K(P)/K(C)/K$.

Let $\{P,R\}$ be a basis of $E[\ell]$ and define
$\rhobar_{E,\ell}\colon G_K\to\GL_2(\F_\ell)$ to be the mod $\ell$
representation attached to $E$ with respect to the basis $\{P,R\}$,
and~$\P\rhobar_{E,\ell}$ to be the associated projective
representation.  Let $B$ denote the Borel subgroup
$\left\{(\begin{smallmatrix} * & *\\ 0 & *\end{smallmatrix})\right\}$
  of~$\GL_2(\F_\ell)$, which has index~$\ell+1$, and let~$B_1$ denote
  the subgroup~$\left\{(\begin{smallmatrix} 1 & *\\ 0 &
    *\end{smallmatrix})\right\}$, which is normal of index~$\ell-1$
    in~$B$. Then
\begin{itemize}
	\item[a)] $K(C)$ is the fixed field of $B \cap  \rhobar_{E,\ell}(G_K)\subseteq \GL_2(\F_\ell)$;
	\item[b)] $K(P)$ is the fixed field of $B_1 \cap  \rhobar_{E,\ell}(G_K) \subseteq \GL_2(\F_\ell)$.
\end{itemize}
The isogeny with kernel~$C$ is defined over~$K$ if and only if
$K(C)=K$; that is, if and only if $\rhobar_{E,\ell}(G_K)\subseteq B$.
We say that the prime~$\ell$ is \emph{reducible} for~$E/K$
if~$\rhobar_{E,\ell}(G_K)$ is contained in a Borel subgroup, that is,
if the representation~$\rhobar_{E,\ell}$, or equivalently the
projective representation~$\P\rhobar_{E,\ell}$, is reducible.  Note
that the projective representation is unchanged under quadratic twist;
hence the set of reducible primes only depends on the $j$-invariant
of~$E$, provided that $j(E)\not=0,1728$.

In general, we have the following.

\begin{lem}\label{lem:div}
$[K(P):K(C)]$ divides $\ell-1$, and $[K(C):K]\le\ell+1$.
\end{lem}
\begin{proof}
By Galois theory, $[K(P):K(C)] = [B \cap \rhobar_{E,\ell}(G_K) : B_1
  \cap \rhobar_{E,\ell}(G_K)]$, which divides~$[B:B_1]=\ell-1$.
Similarly, $[K(C):K] = [\rhobar_{E,\ell}(G_K):B \cap
  \rhobar_{E,\ell}(G_K)] \le \ell+1$.
\end{proof}

The set of reducible primes is invariant under isogeny:
\begin{prop}\label{prop:red-isog}
Let $E_1$ and $E_2$ be elliptic curves defined over the number
field~$K$, that are isogenous over~$K$.  Then $E_1$ and~$E_2$ have the
same sets of reducible primes: that is, $E_1$ has an~$\ell$-isogeny
defined over~$K$ if and only if $E_2$ does. If $j(E_1),j(E_2)\notin\{0,1728\}$ then the set of reducible primes is the same for $E_1$ and $E_2$ that are isogenous over $\overline K$.
\end{prop}
\begin{proof}
First assume that $E_1$ and $E_2$ are isogenous over $K$. By \Cref{prop:factor} we may assume that the given isogeny
$\phi\colon E_1\to E_2$ has prime degree~$p$.  Since the dual isogeny has
the same degree, $p$ is reducible for both curves. Let $\ell$ be a
prime not equal to~$p$ that is reducible for~$E_1/K$.  Then $E_1$ has
a cyclic subgroup~$C$ of order~$\ell$ defined over~$K$, and $\phi(C)$
is a subgroup of~$E_2$ also defined over~$K$ and of order~$\ell$,
since the kernel of~$\phi$ has order coprime to~$\ell$.

Now suppose that $j(E_1),j(E_2)\notin\{0,1728\}$ and that $E_1$ and $E_2$ are isogenous over $ \overline K$. By \Cref{lem:twist} there exists a quadratic twist $E_1'$ of $E_1$ such that $E_1'$ and $E_2$ are isogenous over $K$. Since the set of reducible primes is quadratic-twist-invariant it follows by what has already been proved that $E_1$ and $E_2$ have the same set of reducible primes.
\end{proof}

Apart from small primes~$\ell$, elliptic curves over~$\Q$ cannot
acquire $\ell$-isogenies over extensions of odd degree.

\begin{prop}\label{prop:noCM}
Let $E/\Q$ be an elliptic curve without CM, and let~$\ell$ be an odd
prime such that $E$ has no $\ell$-isogenies defined over~$\Q$.  Then
all~$\ell$-isogenies of~$E$ are defined over number fields of even
degree, unless $j(E)=2268945/128$, in which case~$E$ acquires
$7$-isogenies over the cubic field generated by a root of $x^3-5x-5$.
\end{prop}

\begin{proof}
If $\rhobar_{E,\ell}$ is surjective then all $\ell$-isogenies are
defined over fields of degree~$\ell+1$ (which is even), since this is
the index of the Borel subgroups in~$\GL_2(\F_\ell)$.

Otherwise, if $\ell \geq 17$ and $E$ does not have an $\ell$-isogeny
over $\Q$, by the results of \cite{bilu, pbr, mazur2, serre1} the
projective image $\P\rhobar_{E,\ell}(G_K)$ is contained in the
normaliser of a non-split Cartan subgroup, which is the dihedral
group~$D_{\ell+1}$ of order~$2(\ell+1)$. By \cite[Proposition
  1.13]{zyw}, the projective image is either the whole normaliser or
an index $3$ subgroup, hence has order either $2(\ell+1)$ or
$2(\ell+1)/3$.  Since Borel subgroups of $\PGL_2(\F_{\ell})$ have
order~$\ell(\ell-1)$, the intersection of the projective image with
each Borel subgroup is cyclic of order
dividing~$\gcd(\ell+1,\ell-1)=2$, and hence has even index in the
image.

For $\ell=13$, all the possibilities for the image of Galois are now
understood by the results of \cite{curse}. All the possibilities and
the corresponding possibilities for $[\Q(P):\Q]$ for $P\in E[13]$ can
be found in \cite[Table 2]{gn} (where all the non-surjective
possibilities are listed), and we can see that if $E$ does not have a
$13$-isogeny and $\rhobar_{E,13}$ is not surjective, then $[\Q(P):\Q]$
is either $72$ or $96$, implying that $[\Q(C):\Q]$ is a multiple of
$6$ or $8$.

For $\ell=11$, all the possibilities for $\rhobar_{E,\ell}(G_K)$ are
known and from \cite[Table 1]{gn} we deduce that if $E$ has no
$11$-isogeny over $\Q$, then $[\Q(C):\Q]=12$.

For $\ell=3$ and $5$, the results hold when we replace the base field
$\Q$ by a field~$K$ (of characteristic not~$\ell$), provided that
$\sqrt{5}\notin K$ when $\ell=5$.  For all~$\ell\ge3$, the fields of
definition of the $\ell$-isogenies are determined by the roots of a
polynomial~$f(X)$ over~$K$ of degree~$\ell+1$; for example one may
take $f(X)=\Phi_{\ell}(X,j(E))$ where~$\Phi_{\ell}$ denotes the
modular polynomial.  We claim that the discriminant of~$f$ modulo
squares is $\ell^*=\pm\ell$, the sign being taken so that
$\ell^*\equiv1\pmod4$.

Granted the truth of the claim, our assumption that $E$ has no
$\ell$-isogenies defined over~$K$ means that $f$ has no roots in~$K$
itself.  When~$\ell=3$ this means that either $f$ is irreducible or
the product of two irreducible quadratics, so the roots always have
even degree.  When $\ell=5$ we must exclude the possibility that
$f=gh$, where~$g$ and~$h$ are both irreducible of degree~$3$.  In this
case, $g$ and $h$ have the same splitting field, since over any field
for which three different $\ell$-isogenies are defined, all $\ell+1$
are defined (by looking at the action of $\PGL(2,\F_{\ell})$
on~$\P^1(\F_{\ell})$). But then $\disc(f)$ is a square, so $5$ must be
a square in~$K$.

Now let~$\ell=7$.  If $f$ has irreducible factors of odd degree and no
rational roots, it must factor as $f=gh_1h_2$ where $g$ has degree~$2$
and $h_1,h_2$ both have degree~$3$.  (As we saw in the case $\ell=5$,
if there is an irreducible factor of degree~$3$ then its splitting
field contains all the roots of~$f$, so $f$ cannot also have an
irreducible factor of degree~$5$.)  Both $h_1$ and~$h_2$ have the same
splitting field, which contains that of~$g$, so $f$ has Galois group
isomorphic to~$S_3$, and, since $\disc(f)=-7$ (modulo squares) by our
claim, all three factors have discriminant~$-7$ (modulo squares).  Now
the nontrivial elements of the Galois group of~$f$ either have
order~$2$ and act on each set of conjugate roots via a transposition,
fixing one root of each of the cubic factors; or have order~$3$,
acting as $3$-cycles on the cubic roots and fixing both quadratic
roots.  Hence every element of the Galois group of~$f$ fixes at least
one root of~$f$, so $E$ admits $7$-isogenies modulo~$p$ for almost all
primes~$p$, the situation described by Sutherland in detail
in~\cite{SutherlandLocalGlobal}.
By~\cite[Section~3]{SutherlandLocalGlobal}, this is only possible when
$j(E)=2268945/128$.  The polynomial~$\Phi_7(X,2268945/128)$ has two
irreducible factors of degree~$3$ whose roots each generate a cubic
field isomorphic to $\Q[x]/(x^3-5x-5)$, and a quadratic factor which
splits over~$\Q(\sqrt{-7})$.

It remains to prove the claim.  The faithful action
of~$\PGL(2,\F_{\ell})$ on~$\P^1(\F_{\ell})$ determines (after fixing a
labelling of the points on the projective line) an injective group
homomorphism $\pi\colon \PGL(2,\F_{\ell})\to S_{\ell+1}$, the
symmetric group on~$\ell+1$ letters.  Identifying the roots of~$f$
with~$\P^1(\F_{\ell})$, we see that the composite
$\pi\circ\rhobar_{E,\ell}\colon G_K\to S_{\ell+1}$ gives the
permutation action of~$G_K$ on the roots of~$f$.  The image of~$\pi$
is not contained in the alternating group $A_{\ell+1}$, since
if $a$ generates~$\F_\ell^*$ then
$\pi\left(\begin{pmatrix} a & 0 \\ 0 & 1 \end{pmatrix}\right)$ is
an $\ell-1$--cycle, which is odd.  Hence
the preimage under~$\pi$ of~$A_{\ell+1}$ has index~$2$
in~$\PGL(2,\F_{\ell})$, and must therefore be~$\PSL(2,\F_{\ell})$, as
this is the unique such subgroup: for $\ell\ge5$ this follows from the
simplicity of~$\PSL(2,\F_{\ell})$, while for $\ell=3$ the map~$\pi$ is
an isomorphism.  Hence, for all extensions~$L/K$, the image
$\rhobar_{E,\ell}(G_L)$ is contained in~$\PSL(2,\F_{\ell})$ if and
only if $\pi(\rhobar_{E,\ell}(G_L))\subseteq A_{\ell+1}$, which is if
and only if $\disc(f)$ is a square in~$L$ by standard Galois theory.

On the other hand, for $g\in G_K$ we have
$\rhobar_{E,\ell}(g)\in\PSL(2,\F_{\ell})$ if and only if
$\det\rho_{E,\ell}(g)\in(\F_\ell^*)^2$.  Since the determinant
of~$\rho_{E,\ell}$ is the $\ell$'th cyclotomic character, this holds
if and only if $g(\sqrt{\ell^*}) = \sqrt{\ell^*}$.  Thus $\disc(f)$ is
a square in~$L$ if and only if~$\sqrt{\ell^*}\in L$, and so
$K(\sqrt{\ell^*})=K(\sqrt{\disc(f)})$ as claimed.
\end{proof}

\begin{cor}\label{cor:noCM}
Let $E$ be an elliptic curve without CM defined over a number field
$K$ of odd degree~$d$, such that $j(E)=j\in \Q$, and let $\ell\neq
2,7$ be a prime. Then $\ell$ is reducible for~$E/K$ if and only if
$\ell$ is reducible for $E_0/\Q$, for any elliptic curve $E_0/\Q$ with
$j(E_0)=j$.  The same holds for $\ell=7$, unless $j=2268945/128$ and
$K$ contains a root of $x^3-5x-5$.
\end{cor}

\begin{proof}
  Let $E_0$ be any elliptic curve defined over~$\Q$ with~$j(E_0)=j$.
  By \Cref{prop:noCM}, $\ell$ is reducible for $E_0/\Q$ if and only if
  it is reducible for $E_0/K$, which is if and only if it is reducible
  for $E/K$ by the invariance under quadratic twist.
\end{proof}

\begin{rmk}

Let $E/\Q$ be an elliptic curve without CM with square discriminant,
but with trivial $2$-torsion over $\Q$, such as the one with LMFDB
label \lmfdbec{196}{a}{1}. 
Equivalently, the image of the mod 2
representation attached to $E$ is a cyclic group of order 3.  Let $K$ be the
cyclic cubic field over which the $2$-torsion of $E$ is defined. Then
the 3 points of order $2$ are Galois conjugates of each other, and as
each is the generator of a kernel of a $2$-isogeny to a curve $E_i$,
the three curves $E_i$, $i=1,2,3$ are also Galois conjugates. Since
$E$ does not have CM, their $j$-invariants are distinct, and hence not
defined over $\Q$. We conclude that each $E_i$ is a $\Q$-curve, but
not defined over $\Q$.

This example shows that \Cref{cor:noCM} is not true for $\ell=2$ and
$d=3$.

As an example of the exceptional case when $\ell=7$, the elliptic
curve 
\lmfdbec{2450}{i}{1} has $j$-invariant~$2268945/128$, and has no
$7$-isogenies over $\Q$, but has $7$-isogenies defined over each of
the conjugate cubic fields generated by roots of $x^3 - 5x - 5$, and
two defined over $\Q(\sqrt{-7})$.
\end{rmk}

\begin{defn}
We define $J_\Q(d)$ to be the set of prime numbers $\ell$ for which there exists an $\ell$-isogeny of a
non-CM elliptic curve with $\Q$-rational $j$-invariant without CM over a number field of
degree $d$. Define $I_\Q(d)$ to be the union of all $J_\Q(k)$, $k\leq d$.

\end{defn}

Results about $I_\Q(d)$ can be found in \cite{naj}. When $d=1$, the set
of set of prime degrees of isogenies of non-CM elliptic curves defined
over $\Q$ is
$$J_\Q(1)=I_\Q(1)=\{2,3,5,7,11,13,17,37\},$$
by \cite{mazur2}.

We can now prove \Cref{tm:bounds}~a), that the only primes~$\ell$
that are reducible for a non-CM $\Q$-curve defined over a number
field of odd degree are those that are in $J_\Q(1)$.
\begin{proof}[Proof of \Cref{tm:bounds} a)]
  By \Cref{odd_degree}, the curve $E$ is isogenous to an elliptic
  curve $E'$ with $j(E')\in\Q$, and $\ell$ is also reducible
  for~$E'/K$ by \Cref{prop:red-isog}.  By \Cref{cor:noCM}, $\ell$ is
  reducible for any curve~$E_0/\Q$ with $j(E_0)=j(E')$, and hence
  $\ell$ is one of the primes in $J_\Q(1)$.
\end{proof}

\begin{rmk}
As can be seen in \cite[Theorem 1.2]{BN2}, all elliptic curves with $\Z/2\Z \times \Z/14\Z$-torsion over cubic fields are base changes of elliptic curves defined over $\Q$. Hence there exists a $\Q$-curve that is $2$-isogenous to such a curve with a $28$-isogeny over that cubic field. But there are no elliptic curves with $28$-isogenies over $\Q$. This shows that the restriction to prime degree isogenies in \Cref{tm:bounds} a) is necessary, as there exist composite isogeny degrees between $\Q$-curves that appear over odd degree number fields but not over $\Q$.
\end{rmk}

In the proof of the next Proposition, and also later (\Cref{S-lemma}),
we will need some elementary properties of the \emph{affine linear
  group} $\AGL_1(\F_\ell)$, which we identify with a subgroup
of~$\GL_2(\F_\ell)$:
\[
\AGL_1(\F_\ell) = \left\{\begin{pmatrix} a & b \\ 0 & 1 \end{pmatrix}
\mid a\in\F_\ell^*, b\in\F_\ell\right\}.
\]
\begin{lem}\label{lem:AGL}
The group $\AGL_1(\F_\ell)$ has the following properties:
\begin{enumerate}
\item It acts transitively on~$\F_\ell$ via $x\mapsto ax+b$.
\item It has order~$\ell(\ell-1)$, and is the semidirect product of
  the additive group~$\F_\ell$ and the multiplicative
  group~$\F_\ell^*$, fitting in the short exact sequence
\[
0 \longrightarrow \F_\ell \longrightarrow \AGL_1(\F_\ell)
\longrightarrow \F_\ell^* \longrightarrow 1,
\]
the inner maps being $b\mapsto\begin{pmatrix} 1 & b \\ 0 &
1 \end{pmatrix}$ and $\begin{pmatrix} a & b \\ 0 & 1 \end{pmatrix}
\mapsto a$.
\item For every subgroup~$H\leq\AGL_1(\F_\ell)$, one of two cases
  occurs: either $H$ has order divisible by~$\ell$, acts transitively
  on~$\F_\ell$ and has a subgroup of index~$\ell$; or is cyclic of
  order dividing~$\ell-1$, acts on~$\F_\ell$ with orbits of size
  dividing~$\ell-1$, with at least one fixed point.
  \end{enumerate}
\end{lem}
\begin{proof}
Elementary.  In (3) if $H$ is nontrivial and cyclic of order dividing
$\ell-1$ and generated by $g=\begin{pmatrix} a & b \\ 0 &
1 \end{pmatrix}$, then $a\not=1$ and $b/(1-a)$ is a fixed point of the
action of $g$ (and hence $H$).
\end{proof}

\begin{prop}\label{prop:div}
Let $\ell$ be a prime, $E/K$ be an elliptic curve over a number field
and $C$ a cyclic subgroup of $E$ of order $\ell^n$ for some integer
$n\geq 2$. Then $[K(C):K(\ell C)]$ either equals~$\ell$ or divides~$\ell-1$.
\end{prop}
\begin{proof}
Fix a basis for~$E[\ell^{\infty}]$ and let $\rho_{E,\ell}\colon G_K \to
\GL_2(\Z_{\ell})$ be the $\ell$-adic representation attached to~$E/K$.
We may assume that the basis is chosen so that
$\rho_{E,\ell}(G_{K(C)}) \subseteq G_n$ and $\rho_{E,\ell}(G_{K(\ell
  C)}) \subseteq G_{n-1}$, where for $m\ge0$ we define
\begin{equation*}
G_m = \left\{\begin{pmatrix} a & b \\ \ell^m c & d \end{pmatrix} \right\} \subset \GL_2(\Z_{\ell}).
\end{equation*}
Note that, as $n\ge2$, $G_n$ has index~$\ell$ in~$G_{n-1}$.  Write
$\rho=\rho_{E,\ell}$ and $H_m=\rho(G_K)\cap G_m$ for $m=n,n-1$.  Since
$\ker\rho\subseteq G_{K(C)}$ (a Galois automorphism which fixes all
$\ell$-power torsion points fixes all the points of $C$ and hence
certainly fixes $K(C)$), we have
\[
[K(C):K(\ell C)] = [G_{K(\ell C)}:G_{K(C)}] = [\rho(G_{K(\ell
    C)}):\rho(G_{K(C)})] = [\rho(G_K)\cap G_{n-1}:\rho(G_K)\cap G_{n}] =
[H_{n-1}:H_{n}].
\]
We must show that this index is either $\ell$ or a divisor
of~$\ell-1$.

The map $\pi\colon G_{n-1}\to\AGL_1(\F_\ell)$ defined by $\begin{pmatrix}
  a & b \\ \ell^{n-1} c & d \end{pmatrix} \mapsto
\begin{pmatrix}
  \overline{d}/\overline{a} & \overline{c}/\overline{a} \\ 0 &
  1 \end{pmatrix}$ is a surjective group homomorphism, so that
$G_{n-1}$ acts transitively on $\F_\ell$ via $\begin{pmatrix} a & b
  \\ \ell^{n-1} c & d \end{pmatrix}\colon x \mapsto
(\overline{d}x+\overline{c})/\overline{a}$.  Under this action, the
stabiliser of~$0$ is~$G_n$. (Here the bar denotes reduction
modulo~$\ell$, and $a$ is a unit in~$\Z_{\ell}$ since $n\ge2$).  Hence
the index of $H_n=G_n\cap H_{n-1}$ in~$H_{n-1}$ is the size of the
orbit of $0$ under~$H_{n-1}$.  By \Cref{lem:AGL}(3) the orbit size is
either~$\ell$ or is a divisor of~$\ell-1$, as required.
\end{proof}

We will need a result about degrees of fields of definition of
isogenies that is, in certain instances, more precise than
\Cref{prop:div}. Before stating and proving it, we introduce a
definition.

\begin{defn}
  We say that the $\ell$-adic representation $\rho_{E,\ell}\colon G_K
  \rightarrow \GL_2(\Z_\ell)$ of $E$ is \textit{defined modulo
    $\ell^n$} if the image $\rho_{E,\ell}(G_K)$ contains the kernel of
  the reduction map $\GL_2(\Z_\ell) \to \GL_2(\Z_\ell/\ell^n\Z_\ell)$.
\end{defn}

\begin{prop}
\label{prop:divisibility}
Let $E$ be an elliptic curve defined over a number field $K$ such that
its $\ell$-adic representation is defined modulo $\ell^{n-1}$ for some
$n\ge1$. Then for any cyclic subgroup $C$ of $E(\overline K)$ of order
$\ell^{n}$, we have $[K(C):K(\ell C)]=\ell$.
\end{prop}
\begin{proof}
In the notation of the proof of \Cref{prop:div}, the image
of $H_{n-1}$ in the affine group contains elements of order~$\ell$,
and hence acts transitively.
\end{proof}

Let $G_E(\ell^{n})$ be the reduction modulo~$\ell^n$ of the image of
the $\ell$-adic representation attached to~$E$, that is, of the
composite $G_K\to\GL_2(\Z_{\ell})\to\GL_2(\Z/\ell^n\Z)$.

\begin{rmk}
\label{remark:maarten}
The result of \Cref{prop:div} is best possible, in the sense that
there exist $E/\Q$ and cyclic subgroups $C\in E(\overline \Q)$ of order
$\ell^2$ such that each of the cases $[\Q(C):\Q(\ell C)]=\ell$ and
$[\Q(C):\Q(\ell C)]=\ell-1$ occur. The first case is generic, and
occurs for any cyclic subgroup~$C$ of $\ell$-power order when the
$\ell$-adic representation attached to $E$ is surjective.

The case when $[\Q(C):\Q(\ell C)]=\ell-1$ occurs for example when
$G_E(\ell^2)$ is the reduction of $\Gamma_0(\ell^2)\cap
\Gamma_1(\ell)$ mod $\ell^2$. Then there is one cyclic subgroup $C$ of
$E$ order $\ell^2$ defined over $\Q$. The other $\ell-1$ cyclic
subgroups of order $\ell^2$, which are solutions of the equation (in
groups) $\ell X = \ell C$, form a single orbit under the action of
$\GQ$, and hence are defined over an extension of degree~$\ell-1$. An
example in the case~$\ell=5$ is the elliptic curve with LMFDB
label~\lmfdbec{11}{a}{3}.
\end{rmk}

If we impose additional conditions on the degree of the number field,
then we get an absolute bound on the possible degrees of isogenies.

\begin{prop} \label{prop:nogrowthisog}
Let $d$ be an odd integer not divisible by any $\ell\in J_\Q(1)$. Let $E$
be a $\Q$-curve without CM over a number field $K$ of degree $d$ and
$\phi\colon E \rightarrow E'$ a cyclic isogeny of degree $n$. Then $E$ is
isogenous to an elliptic curve $E''/\Q$ which has a cyclic $n$-isogeny
over $\Q$, and in particular $n\leq 37.$
\end{prop}
\begin{proof}

By \Cref{cor:odd_degree} we conclude that $E$ is isogenous to an elliptic curve $E'$ with $j(E')\in \Q$. By \Cref{cor:noCM}, there exists an elliptic curve $E''/\Q$ with $j(E'')=j(E')$ such that the degrees of isogenies of $E'$ and $E''$ are the same.


Since $d$ is not divisible by any prime $\ell \leq 17$, by
\Cref{lem:div} we have that $E''$ does not gain any $\ell$-isogenies
over $K$
for $\ell \leq 7$. From \Cref{prop:noCM}, we conclude that $E''$ does
not gain any $\ell$-isogeny for $\ell>7$.

Finally, the $\ell$-power degrees of isogenies of $E''$ do not change when extending to $K$ by \Cref{prop:div}.
\end{proof}

\begin{proof}[Proof of \Cref{tm:bounds} b)]
  This follows directly from \Cref{prop:nogrowthisog}.
\end{proof}

\begin{proof}[Proof of \Cref{bound:isogenies}]
We must show that for each odd~$d$, the degrees of isogenies of all
$\Q$-curves over all number fields of degree $d$ are bounded.

By the theory of complex multiplication, there are finitely many
orders $\OO$ of quadratic imaginary fields such that elliptic curves
with CM by $\OO$ are defined over a number field of degree $d$. For
elliptic curves with CM the result now follows from \cite[Theorem
  5.3. b)]{bc}. In the assumptions of this theorem, there is the
condition that $K$ does not contain the field of definition of an
elliptic curve with CM by an order of conductor divisible by $\ell$, but if
this happens then there is an isogeny from $E$ to an elliptic curve
$E'$ with CM by the ring of integers of the CM field of $E$; as $E'$ certainly satisfies the conditions of \cite[Theorem
  5.3. b)]{bc} its possible degrees of $\ell$-power isogenies are bounded so it follows that
the degrees of the possible $\ell$-power isogenies on $E$ over $K$ are also bounded.

Suppose now that $E$ does not have CM. By \Cref{cor:odd_degree} we know $E$ is isogenous to an elliptic curve $E'$ with $j(E')\in \Q$. By \Cref{cor:noCM}, there exists an elliptic curve $E''$ with $j(E'')=j(E')$ such that the degrees of isogenies of $E'$ and $E''$ are the same. We will bound the possible degree of an isogeny of $E''$ from which a bound on the possible degree of an isogeny of $E$ immediately follows.

By \Cref{prop:noCM} we have
that the primes $\ell$ such that there exist $\ell$-isogenies over
number fields of degree $d$ are bounded (by $37$). It remains to show
that the $\ell$-power degrees of isogenies are bounded.

Let $N$ be the degree of a cyclic isogeny of $E''$ over a number field $K$
of degree $d$. Then, by \cite[Theorem 1.2]{arai}, there exists
$B_\ell$ such that for all non-CM elliptic curves defined over $\Q$
the $\ell$-adic representation of $E''$ is defined modulo $\ell^m$ for
some $m\leq B_\ell$. Now from \Cref{prop:divisibility}, we can
conclude that if $E''$ has an isogeny of degree $\ell ^k$ for $k>
B_\ell$, then $d$ is divisible by $\ell^{k-B_\ell}$. It follows that
$\nu_\ell(N) \leq \nu_\ell(d)+B_\ell.$
\end{proof}
\begin{rmk}
If one knows all the $B_\ell$ in the proof of the previous theorem, then $C_d$ can be effectively computed. Using the notation above, if $E''/\Q$ has an $\ell$-power isogeny of degree at most $\ell^a$ over a  number field $K$ of odd degree, then it follows that $E$ has an $\ell$-power isogeny of degree at most $\ell^{2a}$ over $K$, coming from the possible isogeny diagram
$$E \xleftarrow{\ell^a}E''\xrightarrow{\ell^a} E^\sigma.$$

For elliptic curves without CM, we get the bound:
$$C_d^{nonCM}:=\prod_{\ell \in J(\Q)} \ell^{2(\nu_\ell(d)+B_\ell)}.$$
For elliptic curves with CM, one can produce a bound $C_d^{CM}$ using
the results of \cite{bc}.  Finally, set $$C_d:=\max\{C_d^{nonCM}, C_d^{CM} \}.$$
\end{rmk}

\begin{rmk} \label{remark:unbounded}
It is not possible to bound the size of the prime power torsion, and hence the degree of prime power
degree isogenies, over the union of all number fields of degree not divisible by some finite set of primes. To see this, let $E$ be over $\Q$ a non-CM elliptic
curve with a $\Q$-rational point of order $\ell=3,5$ or $7$ for which
the image of the $\ell$-adic representation is as large as the point of
order $\ell$ allows (this is the generic case, so infinitely many
elliptic curves will satisfy this). Then $E$ will have a point of
order $\ell^{n+1}$ over a number field of degree $\ell^{2n}$, as can
be seen from \cite[Proposition 2.2]{gn2}.
\end{rmk}

\section{Torsion of $\Q$-curves over odd degree number fields} \label{sec:torsion}

The following three sets of finite
abelian groups are defined (see \cite{km,gt,guz}) for each positive integer $d$:
\begin{itemize}
     \item $\Phi(d)$ is the set of all possible torsion groups of elliptic curves over number fields of degree $d$.
     \item $\Phi_\Q(d)$ is the set of all possible torsion groups of elliptic curves defined over $\Q$ base-changed to number fields of degree $d$.
     \item $\Phi_{j\in \Q}(d)$ is the set of all torsion groups of
       elliptic curves  over number fields of degree $d$, with rational $j$-invariant.
   \end{itemize}
Thus, $\Phi_{\Q}(d) \subseteq \Phi_{j\in\Q}(d) \subseteq \Phi(d)$.
By Mazur's theorem \cite{mazur}, we have that
\begin{equation}
\label{eqn:MazurList}
\Phi_{j\in \Q}(1)=\Phi_{ \Q}(1)=\Phi(1)=\{\Z/n\Z \text{ for }n=1,
\ldots 10, 12\}\cup\{ \Z/2\Z\times \Z/2m\Z\text{ for }m=1,\ldots, 4\}.
\end{equation}

\begin{prop} \label{S-lemma}
Let $d$ be an odd integer not divisible by any prime $\leq 7$. For all non-CM $\Q$-curves defined over number fields $K$
  of degree $d$, if $E(K)$ has a point of prime order $\ell$ then
  $\ell\in\{2,3,5,7\}$.
\end{prop}
\begin{proof}
Suppose that $E/K$ is a $\Q$-curve without CM over a number field $K$
of degree $d$, and that $E(K)$ has a point of prime order $\ell$. Then
by \Cref{cor:odd_degree} we have that $E$ is isogenous over $K$ to
an elliptic curve $E_0$ with $j(E_0)\in \Q$. There exists a quadratic twist $E'$ of $E_0$ such that $E'$ is defined over $\Q$.
 Since $E(K)$ has a point of
order $\ell$, by \Cref{prop:red-isog}, $E_0$ has an $\ell$-isogeny over
$K$, and since having an $\ell$-isogeny is a quadratic-twist-invariant property, it follows that $E'$ has an $\ell$-isogeny over $K$.
By \Cref{prop:noCM} it follows that either $\ell\in\{2,7\}$ or
$\ell$ is the degree of an isogeny over~$\Q$; so
$\ell\in\{2,3,5,7,11,13,17,37\}$.

It remains to show that
$\ell\not=11,13,17,37$, so suppose that $\ell$ is one of these
primes. Let $K_0$ be the field of definition of an isogeny between $E$ and $E'$; it is either $K$ or a quadratic extension of $K$.

By \cite[Tables 1 and 2]{gn} we see that $E'(K_0)$ cannot have
a point of order $\ell$ over a number field of degree $2d$ where $d$ is not divisible by any prime $\leq 7$.  Then the fact that $E(K_0)$ has a point of order $\ell$, together with \cite[Proposition 1.4]{reveter}, implies that, up
to conjugation, we have
$$H\colon=\overline \rho_{E',\ell}(G_{K_0})=\begin{pmatrix} \chi_\ell(G_{K_0}) & *\\ 0 &1
\end{pmatrix},
$$ with $*$ nonzero, so $H$ is a subgroup of $\AGL_1(\F_\ell)$. Observe that $x\in\F_\ell$ is a fixed point of $H$ with respect to its action on $\F_\ell$ if and only if $\begin{pmatrix} x\\ 1 \end{pmatrix}\in \F_\ell^2$ (viewed as an element of $E'[\ell]$) is a fixed point with regard to the action of $H$ on $E'[\ell]$. Since $E'(K_0)$ has no points of order $\ell$, it follows that $H$ has to act without fixed points on $E'[\ell]\setminus \{0\}$, so $H$ acts without fixed points on $\F_\ell$. By \Cref{lem:AGL} (3) we conclude that $H$ has a subgroup $S$ of index $\ell$.

It then follows from Galois theory that the fixed field $K'$ of $S$ is of degree $\ell$ over $K_0$ and such that $\overline \rho_{E',\ell}(G_{K'})=S.$ By the same argument as before we conclude that the action of $\overline \rho_{E',\ell}(G_{K'})$ (and hence $G_{K'}$) on $E'[\ell]$ has a fixed point so $E'$ has a point of order $\ell$ over $K'$.

But then $E'$ has a point of order $\ell$ over $K'$, of degree $\ell d$ or $2\ell d$, where $d$ is not divisible by any prime $\leq 7$, again contradicting
\cite[Tables 1 and 2]{gn}.
\end{proof}

Recall the following result of Gužvić.

\begin{thm}[{\cite[Theorem 1.1]{guz}}]\label{j-invariants}
Let $p$ be a prime $\geq 7$. Then if $E$ with $j(E)\in \Q$ has a point of order $n$ over a number field of degree $p$, then $\Z/n\Z \in\Phi(1)$.
\end{thm}

We now prove a slightly stronger result.
\begin{thm}\label{thm:div7}

  Let $d$ be an odd integer not divisible by any prime $\leq 7$. Then $\Phi_{j\in \Q}(d)=\Phi_\Q(d)=\Phi_\Q(1)$.
\end{thm}
\begin{proof} First note that $\Phi_\Q(d)=\Phi_\Q(1)$ by
\cite[Corollary 7.3]{gn}. Hence we need to show that $\Phi_{j\in\Q}(d)
= \Phi_\Q(d)$.  Certainly $\Phi_{j\in\Q}(d) \supseteq \Phi_\Q(d)$, so
we must show that if $E$ is an elliptic curve over $K$, a number field
whose degree $d$ satisfies the stated condition, with $j(E)\in \Q$,
then the torsion subgroup of $E(K)$ also occurs as the torsion
subgroup of $E_0(K)$ where $E_0$ is an elliptic curve defined
over~$\Q$.

If $E$ has CM, then the result follows by \cite[Theorem 1.2]{bp}, so
suppose now that $E$ does not have CM.

Let $E_0$ be an elliptic curve defined over~$\Q$ with $j(E_0)=j(E)$,
so $E_0$ is a quadratic twist $E_0$ of $E$, and we have $E\simeq
E_0^\delta$ for some $\delta \in (K^*)/(K^*)^2$.

We first prove that if $E(K)$ has a point of order $n$, then $n$
already occurs as the order of an element of a group in
$\Phi_{\Q}(d)$.  Assume the opposite, i.e $E(K)$ has a point of order
$n$ but no group in $\Phi_\Q(d)$ has a point of order $n$. We will
derive a contradiction by showing that $E$ is the base-change of a
curve defined over $\Q$.

Since $E_0\simeq E$ over $K(\sqrt\delta)$, it follows that
$E_0(K(\sqrt\delta))$ has a point of order $n$. Since there are no
elements of order $n$ in any group in $\Phi_\Q(1)$ (since
$\Phi_\Q(1)\subseteq\Phi_\Q(d)$), it follows that $E_0(\Q)$ does not
have a point of order $n$, so the torsion of $E_0$ grows from $\Q$ to
$K(\sqrt \delta)$. In particular, there exists a prime $\ell\mid n$ such
that $\#E_0(K(\sqrt\delta))[\ell^\infty]>\#E_0(\Q)[\ell^\infty]$.

If $\#E_0(\Q)[\ell]>1$, then $\ell\leq 7$ and a point $P\in E_0(\Q)$
of order $\ell^k$ for some~$k\geq 1$ becomes divisible by $\ell$ in
$E_0(K(\delta))$, i.e., there exists an $P'\in E_0(K(\sqrt\delta))$
such that $\ell P'=P$. From \cite[Proposition 4.6]{gn} we see that
$[\Q(P'):\Q]$ has to divide $\ell^2(\ell-1)$. Since $[K:\Q]=d$, it follows that $[\Q(P'):\Q]$ divides $2d$. Hence $[\Q(P'):\Q]$ divides $\gcd(\ell^2(\ell-1), 2d)$ and since $d$ is not divisible by any prime $\leq 7$, we conclude that $[\Q(P'):\Q]$ divides $2$. It follows that the only possibility
is that $\Q(P')$ is a quadratic field.

Otherwise, $E_0(\Q)[\ell]$ is trivial, and there exists $P'\in
E_0(K(\sqrt{\delta}))$ of order $\ell$. By \cite[Theorem 5.8]{gn}, we
see that $[\Q(P'):\Q]$ has to be divisible by $4$ for $\ell\geq 17$. For $\ell \leq 13$, we see that $[\Q(P'):\Q]$ is never of the form $2t$, with $t>1$ divisible by primes $\geq 7$. Hence it again follows that $\Q(P')$
is a quadratic field.

Thus $\Q(P')=\Q(\sqrt{\delta_0})$ for some $\delta_0\in\Q^*$, and the
$\ell$-power torsion growth occurs over the quadratic
field~$\Q(\sqrt{\delta_0})$. As $d=[K:\Q]$ is odd,
$\Q(\sqrt{\delta_0})\nsubseteq K$; but
$\Q(\sqrt{\delta_0})=\Q(P')\subseteq K(\sqrt{\delta})$, so it follows
that $K(\sqrt{\delta_0})=K(\sqrt{\delta})$. Hence
$\delta\alpha^2=\delta_0$ for some $\alpha \in K^*$. It follows that
$E_0^{\delta_0}\simeq E_0^\delta\simeq E$ over~$K$; in other words,
$E$ is a base change of $E_0^{\delta_0}$, which is an elliptic curve
defined over $\Q$, giving us a contradiction.

So far we have shown that the cyclic subgroups of groups in $\Phi_{j\in \Q}(d)$ are
all also subgroups of groups in $\Phi_{\Q}(d) = \Phi_{\Q}(1)$; in particular, the orders
of elements of groups in $\Phi_{j\in \Q}(d)$ are ${}\le12$. It remains
to check that there are no noncyclic groups in $\Phi_{j\in \Q}(d)$
that are not in $\Phi_{\Q}(d)$. Since a number field of odd degree has
no roots of unity apart from $\pm 1$, the only remaining possibilities
are $\Z/2\Z \times \Z/2n\Z$.  For $1\le n\le 4$ these are already
in~$\Phi_{\Q}(d)$, while for $n>6$ they contain elements of
order~${}>12$, so cannot occur; we are then left with the cases $n=5$
and $n=6$. So suppose $E(K)_{\tors} \simeq \Z/2\Z \times \Z/2n\Z$ with
$n=5$ or~$6$. Now $E_0(K(\sqrt \delta))$ has a subgroup isomorphic to
$\Z/2\Z \times \Z/2n\Z$. Let $n'=5$ or~$3$, respectively, and $P'$ a
point of order $n'$ in $E_0(K(\sqrt \delta))$ (where $E_0\simeq
E^{\delta}$ is defined over~$\Q$ as before).  Using the same argument
as before, we conclude that $\Q(P')$ is a quadratic field and that $E$
is a base change of an elliptic curve defined over $\Q$, completing
the proof.
\end{proof}

\begin{proof}[Proof of \Cref{bound:torsion}]
Let $E$ be a $\Q$-curve over a number field $K$ of degree $p$. If $E$ has CM, the result follows by \cite[Theorem 1.4]{bcs}, so suppose from now on that $E$ does not have CM.

By \Cref{odd_degree} it follows that $E$ is isogenous to a curve $E'$
defined over $\Q$. Let $\phi\colon E'\rightarrow E$ be this isogeny (which is defined over $\Qbar$) and define $n:=\deg \phi$. We factor $\phi=\phi_2 \circ\phi_1$, where $\deg\phi_1$ is divisible only by primes $\leq 7$ and $\deg\phi_2$ is divisible only by primes $\geq 11$. We have the diagram of isogenies (over $\Qbar$)
$$E'\xrightarrow{\phi_1}E''\xrightarrow{\phi_2}E.$$

 Let $\ell\leq 7$ be a divisor of $n$. As the degree of $K$ is, by assumption, coprime to $\#\overline\rho_{E',\ell}(G_{\Q})$ (which is a divisor of $\ell(\ell-1)^2(\ell+1)$ since $\overline\rho_{E',\ell}(G_{\Q})$ is a subgroup of $\GL_2(\F_\ell)$) for all $\ell \leq 7$, we conclude that $K\cap \Q(E'[\ell])=\Q$ so $\overline\rho_{E',\ell}(G_{\Q})=\overline\rho_{E',\ell}(G_{K}),$
and in particular $E'$ has the same $\ell$-isogenies that it had over $\Q$. This implies that $\phi_1$ is defined over $\Q$ and hence we have $j(E'')\in \Q$.

Now since $E''$ and $E$ are isogenous over $\Qbar$, by \Cref{lem:twist}, there is a twist of $E''^\delta$ of $E''$ that is isogenous over $K$ to $E$. Since $E''^\delta$ is isogenous over $K$ to $E$ by an isogeny of degree coprime to $E(K)_{\tors}$ and $E''^\delta(K)_{\tors}$, it follows that $E(K)_{\tors}\simeq E''^\delta(K)_{\tors},$ and hence $E(K)_{\tors}\in \Phi_{j \in \Q}(p)=\Phi(1)$, with the last equality following from \Cref{thm:div7}.
\end{proof}

\section{A $\Q$-curve testing algorithm}
\label{sec:algorithm}

In this section we give an algorithm for testing whether a given
elliptic curve~$E$, defined over a number field~$K$, is a $\Q$-curve,
or equivalently, whether a given algebraic number~$j$ is a
$\Q$-number.  We assume that we know how to test whether~$E$ is CM,
and also that we can compute the $K$-isogeny class of~$E$---though, as
we will see, only a weak form of the latter is needed.  We start with
a brief discussion of these two questions.  The main algorithm then
proceeds by applying a series of straightforward tests for necessary
conditions satisfied by $\Q$-curves, which in practice quickly allow
us to return an answer of ``no'' for all non-$\Q$-curves, followed by
tests for sufficient conditions, allowing us to return the answer
``yes'' for genuine $\Q$-curves.

The trivial first steps are to return ``yes'' if $j(E)\in\Q$, and
otherwise to replace $E$ by a curve defined over~$\Q(j)$ in case
$\Q(j)$ has degree strictly less than~$K$.

\subsection{Testing for CM}
\label{subsec:cm_test}
Both \Sage\ and \Magma\ have functions for testing whether an
algebraic number is a CM $j$-invariant.  We therefore say no more
about such tests here, except to remark that for each degree~$h$ there
are only finitely many CM $j$-invariants of degree~$h$, and complete
lists of these have been computed for all $h\le100$, so that for small
degrees the tests just amounts to checking that $j$ is an algebraic
integer whose minimal polynomial is in a precomputed list.  For
example, the total number for $h\le10$ is $705$.  See~\cite{watkins}
for a list of all fundamental negative discriminants with class
number~$h\le100$, and for the extension including non-fundamental
discriminants, see~\cite{klaise}; given a discriminant~$D<0$, the
Hilbert Class Polynomial~$H_D(X)$ is the minimal polynomial of the
associated $j$-invariants.

\subsection{Computing the $K$-isogenous $j$-invariants (for non-CM curves)}
\label{subsec:isog_js}

Given a non-CM elliptic curve~$E$ defined over a number field~$K$, the
problem of computing the complete (finite) set of $K$-isomorphism
classes of elliptic curves isogenous to~$E$ over $K$ can be divided
into three steps: first, determine the finite set of reducible
primes~$\ell$; next, compute all curves $\ell$-isogenous to~$E$
(over~$K$) for each such~$\ell$; finally, iterate until each curve
encountered is isomorphic to one already in the list.  This procedure
is implemented in \Sage, though not yet in \Magma\ (except over~$\Q$).
For our purposes, it suffices to use a simpler (and faster) procedure
that outputs a list of isogenous $j$-invariants that is not
necessarily complete, which we describe briefly here.

\subsubsection{Determining the reducible primes}
The difficult problem of computing the complete finite set of
reducible primes for a non-CM elliptic curve~$E$, has been solved both
by Larson and Vaintrob in \cite{larson} and also by Billerey in
\cite{billerey}; both methods are implemented in \Sage%
~(see \cite{sage-larson-code}).  For our application to $\Q$-curve
testing, however, it will suffice to only implement a simple necessary
test for reducibility.  If $\ell$ is reducible, then for all
primes~$\p$ of good reduction, the integer $a_{\p}(E)^2-4N(\p)$ must
be a square (possibly zero) modulo~$\ell$, since it is the
discriminant of the characteristic polynomial of
$\rhobar_{E,\ell}(\Frob\p)$, whose eigenvalues lie
in~$\F_\ell$. (Here, $a_{\p}(E)$ denotes the trace of Frobenius of the
reduction of~$E$ at~$\p$, whose reduction modulo~$\ell$ is the trace
of $\rhobar_{E,\ell}(\Frob\p)$, while $\det\rhobar_{E,\ell}(\Frob\p)
\equiv N(\p)\pmod{\ell}$.)  Computing these integers for all~$\p$ with
norm less than some bound, we can discard any primes~$\ell$ modulo
which any of the integers is a quadratic non-residue.  Using this, it
is very fast to determine a small set of primes $\ell$ up to some
bound containing all reducible primes (and possibly some others) up to
that bound.  The harder, and more time-consuming, part is to prove
that there are no reducible primes greater than the bound chosen, but
this is not needed for the $\Q$-curve test.

The \Sage\ command {\tt E.reducible\_primes()} will, if $E$ is not CM,
return a provably complete list of the reducible primes~$\ell$
for~$E$, while {\tt E.reducible\_primes(algorithm='heuristic',
  max\_l=B)} returns the set of reducible primes less than~$B$.

\subsubsection{Computing $\ell$-isogenous $j$-invariants}
For any fixed prime~$\ell$ we may compute the $j$-invariants of
elliptic curves $\ell$-isogenous to~$E$ over $K$ by finding the roots
in~$K$ of $\Phi_{\ell}(X,j(E))$.  Using the methods of~\cite{BLS} it
is practical to compute $\Phi_{\ell}(X,Y)$ for primes $\ell<5000$, and
a precomputed database for~$\ell\le300$ is available at
\cite{SutherlandDatabase}.  By comparison, the largest reducible prime
for any non-CM elliptic curve in the LMFDB database is currently
(September 2020)~$\ell=41$, which occurs for four isogeny classes
over~$\Q(\sqrt{-1})$, for example the class with
label~\lmfdbecnfiso{2.0.4.1}{84050.1}{b}.

Iterating this process, we can compute, starting from any $j\in K$, a
possibly complete list of isogenous $j'\in K$ linked to~$j$ by
isogenies of degrees supported on primes~$\ell\le B$ for some
bound~$B$.  This list is either known to be complete, if for $B$ we
take a rigorous bound such as provided by Billerey's algorithm, or
not, and the $\Q$-curve testing algorithm will take this into account.

\subsection{Necessary conditions}
\label{subsec:necessary}
Since most elliptic curves are not $\Q$-curves, it is useful and
efficient to have a series of necessary conditions for being a
$\Q$-curve that are easy to check, since obviously if any of these
fail then we know that the curve is not a $\Q$-curve.  We do not
assume that the field of definition, $K$, is Galois over $\Q$.  Our
tests are local, and are divided into those which involve primes of
good and bad reduction respectively.  For a prime~$\p$ of good
reduction we denote by~$E_{\p}$ the reduction of~$E$ modulo~$\p$ and
by~$a_{\p}$ its trace of Frobenius.

\subsubsection{Local tests at good primes}
If the base field~$K$ were Galois and we only needed to test that~$E$
was isogenous over~$K$ itself to all its conjugates, a necessary
condition would be that $a_{\p}(E)=a_{\p'}(E)$ for any two conjugate
primes~$\p$, $\p'$ of~$K$.  We replace this with a condition that is
valid when $K$ is not necessarily Galois and that detects isogeny
over~$\Qbar$.

\begin{prop}
  \label{prop:good-P-test}
Let $E$ be a $\Q$-curve defined over the number field~$K$.  Let $p$ be
a rational prime not dividing the norm of the conductor of~$E$, and
let $\p$, $\p'$ be primes of~$K$ above~$p$.  Then $E$ has good
reduction at both~$\p$ and~$\p'$, and
\begin{enumerate}
\item $E_{\p}$ and $E_{\p'}$ are either both ordinary or both
  supersingular;
\item in the ordinary case, the integers $d_{\p} = a_{\p}(E)^2 -
  4N(\p)$ and $d_{\p'} = a_{\p'}(E)^2 - 4N(\p')$ are both negative and
  have the same square-free part.
\end{enumerate}
\end{prop}
\begin{proof}
The condition of being ordinary or supersingular is invariant under
base change and under isogeny (over finite fields).  In the ordinary
case, the endomorphism algebra is an imaginary quadratic field~$K$,
which is also invariant under isogeny (since the endomorphism rings of
all isogenous curves are orders in~$K$) and under base-change (since
for ordinary curves all endomorphisms are already defined over the
base field).

Take a finite extension $L/K$ that is Galois, and such that all the
isogenies between Galois conjugates of the base-change of~$E$ from~$K$
to~$L$ are defined over~$L$.  Let~$\q$ and~$\q'$ be primes of~$L$
above~$\p$ and~$\p'$ respectively.  Since $\Gal(L/\Q)$ acts
transitively on the primes of $L$ above~$p$, the $\Q$-curve condition
implies that the reductions $E_{\q}$, $E_{\q'}$ are isogenous.  Hence
\[
E_{\p}\ \text{ordinary} \iff
E_{\q}\ \text{ordinary} \iff
E_{\q'}\ \text{ordinary} \iff
E_{\p'}\ \text{ordinary},
\]
giving (1).  Assume that we are in the ordinary case.  Then in the
Hasse bound $\left|a_{\p}(E)\right| \le 2\sqrt{N(\p)}$ we have strict
inequality, so $d_{\p}=a_{\p}(E)^2 - 4N(\p)<0$.  The endomorphism ring
of~$E_{\p}$ is an order in the imaginary quadratic
field~$\Q(\sqrt{d_{\p}})$, and that of~$E_{\p'}$ is a (possibly
different) order in the same field, giving (2).
\end{proof}

\subsubsection{Local tests at bad primes}
Again, if $K$ were Galois and we only needed to check that $E$ was
isogenous over~$K$ to its Galois conjugates, the conditions at bad
primes could be combined into a single test that the conductor of~$E$
is Galois invariant.  We replace this with a condition that is stable
under base-change.

\begin{prop}
  \label{prop:bad-P-test}
Let $E$ be a $\Q$-curve defined over the number field~$K$.  Let $p$ be
a rational prime, and let $\p$, $\p'$ be primes of~$K$ above~$p$.
Then
\[
\ord_{\p}(j(E))<0 \iff \ord_{\p'}(j(E))<0.
\]
\end{prop}
\begin{proof}
The $j$-invariant has negative valuation at~$\p$ if and only if the
reduction at~$\p$ is potentially multiplicative, and this condition is
invariant under base-change.  Assume that $\ord_{\p}(j(E))<0$, and
take an extension~$L/K$ and primes~$\q$, $\q'$ as in the proof of
\Cref{prop:good-P-test}, with the additional condition that the
base-change~$E_L$ of~$E$ to~$L$ has bad multiplicative reduction (not
just potentially multiplicative) at~$\q$.  Let $g\in\Gal(L/\Q)$ be such
that $g(\q)=\q'$.  Since $E_L$ and $g(E_L)$ are isogenous, $E_L$ also
has bad multiplicative reduction at~$\q'$, so $\ord_{\q'}(j(E))<0$ and
hence also $\ord_{\p}(j(E))<0$.
\end{proof}

\subsection{Sufficient conditions}
\label{subsec:sufficient}
Here we show how to prove that a non-CM elliptic curve~$E/K$ is a
$\Q$-curve, using possibly incomplete knowledge of the finite
set~$\JJ$ of $j$-invariants isogenous to~$j(E)$ over~$K$.  It is not
necessary to consider any isogenies defined over extensions of~$K$.
If we know that we have the complete $K$-isogeny class, then this
method can also be used to prove that a curve is not a $\Q$-curve,
though in practice that is more easily done using the methods given
above.

Note that since $E$ does not have CM, the $j$-invariants of the curves
in the $K$-isogeny class of~$E$ are distinct; this follows from
\Cref{lem:cyclic-deg}.

Using the method of \Cref{subsec:isog_js}, we compute a subset
$\JJ_0\subseteq\JJ$, which may be a proper subset.  If any elements
of~$\JJ_0$ are rational, then $E$ is a $\Q$-curve.  In this case the
$\Q$-class of~$j(E)$ is rational (as defined in
\Cref{sec:Q-curve-summary}).

Otherwise, first suppose that we know that $\JJ_0=\JJ$.  Then we are
able to apply \Cref{prop:isog-to-central}, testing the condition that
$E$ is a $\Q$-curve if and only if~$\JJ$ includes a complete set of
Galois conjugates. To this end, we compute the degree and minimal
polynomial of each~$j\in\JJ$, and see whether any of these polynomials
occurs with multiplicity equal to its degree.  It suffices to examine
the $j$-invariants of $2$-power degree, since $E$ is a $\Q$-curve if
and only if this set contains a complete set of Galois conjugates, or
equivalently if and only if some minimal polynomial of $2$-power
degree~$d$ occurs $d$ times in the collection.

On the other hand, suppose that we do not know whether $\JJ_0=\JJ$,
having used a non-rigorous bound in determining the reducible primes
by the method of \Cref{subsec:isog_js}.  We may still test whether
$\JJ_0$ contains a complete conjugacy class, and if this is the case
then $E$ is certainly a $\Q$-curve.  However, if in this
case we do not see a complete conjugacy class in~$\JJ_0$, then we
cannot conclude that $E$ is not a $\Q$-curve, since $\JJ_0$ may be a
proper subset of~$\JJ$.  If this occurs, then we can apply more of the
necessary tests of \Cref{subsec:necessary} to try to prove that $E$ is
not a $\Q$-curve, or compute a rigorous bound on the reducible primes
in order to establish that we do in fact have $\JJ_0=\JJ$.

\subsection{The algorithm}
We summarise the results of this section by providing pseudocode for
our algorithm.  We denote the minimal polynomial of an algebraic
number~$j$ by~$m_j$, the conductor of~$E$ by~$\cond(E)$ and the norm
of an ideal~$\a$ by~$N(\a)$.
\medskip

\noindent\textbf{Algorithm QCurveTest} \\
\textbf{Input:} An elliptic curve $E$ defined over a number field $K$,
and positive integers~$B_1,B_2$.\\
\textbf{Output:} \True\ if $E$ is a $\Q$-curve, else \False.
\begin{enumerate}
\item If $j(E)\in\Q$ then return \True.
\item If $j(E)$ is a CM $j$-invariant then return \True.
\item Set $N=N(\cond(E))$.
\item \label{alg:bad} For each prime $p\mid N$:
\begin{enumerate}
  \item If $\{\ord_{\p}(j(E)):\p\mid p\}$ contains both negative and
    non-negative integers then return \False.
\end{enumerate}

\item \label{alg:good} For each prime $p\nmid N$ with $p\le B_1$:
\begin{enumerate}
\item If $\{E_{\p}:\p\mid p\}$ are all ordinary:
\begin{enumerate}
\item If $\{a_{\p}(E)^2-4N(\p):\p\mid p\}$ do not all have the same
  squarefree part then return \False.
\end{enumerate}
\item Else if $\{E_{\p}:\p\mid p\}$ are not all supersingular then
  return \False.
\end{enumerate}
\item \label{alg:isog} Compute the partial $K$-isogeny class~$\CC$ of~$E$, using a
  bound of $B_2$ on the reducible primes.
\item Compute $\JJ=\{j(E'):E'\in\CC\}$,
  $\PP=\{(j',m_{j'}):j'\in\JJ\}$.
\item If $\JJ$ contains a rational number then return \True.
\item Remove from $\PP$ any pairs~$(j',m')$ with $\deg(m')$ not a
  power of $2$.
\item For each $(j',m')\in\PP$:
  \begin{enumerate}
  \item If $\#\{(j'',m'')\in\PP: m''=m'\}=\deg(m')$ then return
    \True.
  \item Remove $\{(j'',m'')\in\PP: m''=m'\}$ from~$\PP$.
  \end{enumerate}
\item \label{alg:Bill} Compute a bound~$B_2'$ on the reducible primes for $E$ (using
  Billerey's algorithm, for example).
\item \label{alg:exit} If $B_2\ge B_2'$ return \False.
\item Increase $B_1$ by a factor of~$2$, replace $B_2$ by~$B_2'$, and
  go to line~(5).
\end{enumerate}

Note that we loop back to line~\ref{alg:good} at most once. We only
reach line~\ref{alg:Bill} if either $E$ is not a $\Q$-curve but it
passes all the necessary tests in lines~\ref{alg:bad}
and~\ref{alg:good}, or it is a $\Q$-curve but the partial isogeny
class we computed in line~\ref{alg:isog} is too small to contain a
complete set of Galois conjugates.  On repeating line~\ref{alg:isog}
we will certainly have the complete isogeny class, hence if we reach
line~\ref{alg:Bill} a second time we know that $E$ is not a
$\Q$-curve; hence the exit at line~\ref{alg:exit}.  It would be
possible to increase~$B_2$ one or more times first, before replacing
it with a rigorous bound, but we expect that using bounds of~$1000$
for both~$B_1$ and~$B_2$ it is unlikely that line~\ref{alg:Bill} will
be reached at all.

Our \Sage~code implementing this algorithm is available
at~\cite{ecnf}.  The function \texttt{is\_Q\_curve()} takes an
elliptic curve over an arbitrary number field as input, and optionally
returns, as well as a \texttt{True}/\texttt{False} value, a
``certificate'' enabling a simple verification of the correctness of
the output.  For $\Q$-curves~$E$, the certificate consists of either a
CM discriminant if $E$ has CM, or otherwise a quadruple $(r, \rho, N,
H)$ where $r$, $\rho$, and~$N$ are the quantities defined above, and
$H\in\Z[X]$ is an irreducible monic polynomial in~$\Z[X]$ of
degree~$2^\rho$ whose roots are all $j$-invariants of elliptic curves
isogenous to~$E$ over its field of definition.

\appendix

\section{$\Q$-curves and $\Q$-numbers}

Here we give a self-contained account of the theory of $\Q$-curves,
establishing the results stated in \Cref{sec:Q-curve-summary}.
Most of the ideas presented here may be found in Elkies' article ``On
elliptic $K$-curves'' (see \cite{elkies}); however, we have found it
more convenient to present this material in terms of properties of
certain algebraic numbers~$j$, viewed as $j$-invariants of elliptic
curves, in order to prove the additional results we need that are not
in~\cite{elkies}.

We take~$\Q$ to be our base field, and denote by $\Qbar$ its algebraic
closure, the field of algebraic numbers. Set $G=G_{\Q}=\Gal(\Qbar/\Q)$.
Everything here could also be done with an arbitrary base field~$K$,
replacing~$\Qbar$ with a separable closure~$M$ of~$K$ and $G$
by~$\Gal(M/K)$.  Some minor changes would be needed in
case~$\ch(K)\not=0$.

See Silverman \cite{SilvermanBook1} for standard properties of
isogenies used here.

\subsection{Isogenies, degrees and the isogeny graph}

\subsubsection{Isogenous algebraic numbers}
We define an equivalence relation \emph{isogeny} on $\Qbar$ as
follows.  Let $j_1,j_2\in\Qbar$.  Then $j_1$ and $j_2$ are
\emph{isogenous}, denoted $j_1\sim j_2$, if for each pair $E_1,E_2$ of
elliptic curves over~$\Qbar$ with $j(E_i)=j_i$ for $i=1,2$, there is
an isogeny~$\phi\colon E_1\to E_2$ defined over~$\Qbar$.  Clearly this
condition does not depend on the choice of the $E_i$, since different
curves with the same $j$-invariant are isomorphic over $\Qbar$.  We
call the triple~$(E_1,E_2,\phi)$ consisting of such a pair of
curves~$E_i$ and the isogeny between them a \emph{realization} of the
isogeny relation~$j_1\sim j_2$.

Isogeny is an equivalence relation: symmetry comes from the existence of
dual isogenies.  Hence we can partition $\Qbar$ into equivalence
classes called \emph{isogeny classes}.

The property of having complex multiplication (CM) is
isogeny-invariant; it is even true that
$\End(E_1)\otimes\Q\cong\End(E_2)\otimes\Q$ for isogenous curves
$E_1$, $E_2$, though in general $\End(E_1)$ and~$\End(E_2)$ are
different orders in their common field of fractions.  Hence it makes
sense to define an isogeny class as being CM or non-CM, the set of CM isogeny
classes in $\Qbar$ being in bijection with the set of imaginary quadratic fields.
We will be mainly concerned with non-CM isogeny classes.

\subsubsection{Degrees of isogenies}
To each pair of isogenous non-CM algebraic numbers $j_1,j_2$ we assign
a positive integer, the \emph{degree} $\deg(j_1,j_2)$ to be the degree
of a cyclic isogeny $\phi\colon E_1\to E_2$ realizing the relation $j_1\sim
j_2$.  This is well-defined by the following lemma.

\begin{lem}
\label{lem:cyclic-deg}
  Let $E_1$, $E_2$ be isogenous elliptic curves without CM over $\Qbar$.
Then there is a cyclic isogeny $\phi\colon E_1\to E_2$, and it is unique up
to sign.  In particular, the positive integer $d=\deg(\phi)$ is
well-defined as the degree of a cyclic isogeny from $E_1$ to $E_2$.
Every isogeny from $E_1$ to $E_2$ has degree~$dm^2$ for some $m\ge1$,
and is cyclic if and only if $m=1$.
\end{lem}
\begin{proof}
Let $\phi_0\colon E_1\to E_2$ be an isogeny.  Then $\ker(\phi_0)$ is a
finite subgroup of $E_1(\Qbar)$.  Let $m\ge1$ be maximal such that
$\ker(\phi_0)$ contains~$E_1[m]$ (the kernel of the
multiplication-by-$m$ map $E_1\to E_1$).  Then $\phi_0=\phi\circ[m]$
for some cyclic isogeny~$\phi\colon E_1\to E_2$.

Let $d=\deg(\phi)$, and suppose that $\psi\colon E_1\to E_2$ is another
cyclic isogeny, of degree~$d'$.  Then $\hat{\psi}\circ\phi$ is an
endomorphism of~$E_1$ of degree $dd'$. Since $\End(E_1)\cong\Z$ we
have $\hat{\psi}\circ\phi = [\pm n]$ with $n$ a positive integer
satisfying $n^2 = \deg(\hat{\psi}\circ\phi) = dd'$.  Now $\ker(\phi)$
is cyclic of order~$d$ and is a subgroup
of~$\ker([n])\cong(\Z/n\Z)^2$, so
$d\mid n$.  Similarly, $d'\mid n$; hence $d=d'=n$ and $\psi=\pm \phi$.

The last part is clear.
\end{proof}

In terms of the modular polynomials $\Phi_d(X,Y)\in\Z[X,Y]$ we have
$j_1\sim j_2$ with $\deg(j_1,j_2)=d$ if and only if
$\Phi_d(j_1,j_2)=0$.

\begin{cor}
\label{cor:square-deg}
Let $j_1,j_2,j_3\in\Qbar$ be isogenous and not CM.  Then
\[
\deg(j_1,j_3) \equiv \deg(j_1,j_2)\deg(j_2,j_3)\pmod{(\Q^*)^2}.
\]
\end{cor}
\begin{proof}
For $1\le r\le3$ let $E_r$ be an elliptic curve with $j(E_r)=j_r$, and
for $1\le r < s \le3$ let $\phi_{rs}$ be a cyclic isogeny from $E_r$
to~$E_s$.  Then $\phi_{23}\circ\phi_{12}$ and $\phi_{13}$ are both
isogenies from $E_1$ to~$E_3$, so by \Cref{lem:cyclic-deg} their
degrees are the same up to squares.
\end{proof}

\begin{cor}
  \label{cor:squarefree-deg}
  Let $j_1,j_2,j_3\in\Qbar$ be isogenous, not CM, with
  $\deg(j_1,j_2)=\deg(j_1,j_3)$ and $\deg(j_2,j_3)$ square-free. Then
  $j_2=j_3$.
\end{cor}
\begin{proof}
  $\deg(j_2,j_3)$ is a square by \Cref{cor:square-deg} and is
  also square-free, hence is~$1$.
\end{proof}

The next results concern the minimal field of definition of an isogeny
realizing an isogeny relation.  These are certainly well-known: see, for
example, Lemma~3.1 in \cite{nlf}.

\begin{lem}
\label{lem:twist}
  Let $E_1$, $E_2$ be non-CM elliptic curves defined over a number
  field~$K$. If $E_1$ and $E_2$ are isogenous over~$\Qbar$ then there
  exists a twist of~$E_2$ that is isogenous to~$E_1$ over~$K$ itself.
\end{lem}
\begin{proof}
Let $\phi\colon E_1\to E_2$ be an isogeny over $\Qbar$.  For
each~$g\in G_K$, $\phi^g$ is another isogeny~$E_1\to E_2$ of the same
degree as~$\phi$, hence, by \Cref{lem:cyclic-deg}, we have
$\phi^g=\alpha(g)\circ\phi$ with $\alpha(g)=\pm1$.  The map
$g\mapsto\alpha(g)$ is a homomorphism, hence there exists~$d\in K^*$
such that $\alpha(g) = g(\sqrt{d})/\sqrt{d}$.  Then the quadratic
twist of $E_2$ by~$d$ is isogenous to~$E_1$ over~$K$.
\end{proof}

\begin{cor}
  \label{cor:field-of-def}
Let $j_1,j_2\in\Qbar$ be isogenous and not CM.  Then there exists an
isogeny~$\phi\colon E_1\to E_2$ realizing the relation~$j_1\sim j_2$, with
$E_1$, $E_2$ and~$\phi$ all defined over $\Q(j_1,j_2)$.
\end{cor}
\begin{proof}
Take any curves~$E_i$ defined over~$\Q(j_i)$ with $j(E_i)=j_i$
for~$i=1,2$.  By \Cref{lem:twist} with~$K=\Q(j_1,j_2)$, after
replacing~$E_2$ by a twist if necessary, there is an isogeny $E_1\to
E_2$ defined over~$K$.
\end{proof}

The following easy fact will be used repeatedly.
\begin{lem}
  \label{lem:G-invariance}
  Let $j_1,j_2\in\Qbar$.  If $j_1\sim j_2$ then for all~$g\in G$ we
  also have $g(j_1)\sim g(j_2)$, and
  $\deg(g(j_1),g(j_2))=\deg(j_1,j_2)$.
\end{lem}
\begin{proof}
Applying any Galois automorphism to a cyclic isogeny $E_1\to E_2$
gives a cyclic isogeny $g(E_1)\to g(E_2)$ of the same degree.
\end{proof}
For an alternate proof, apply $g$ to the equation~$\Phi_d(j_1,j_2)=0$.

\subsubsection{Factorization of isogenies and Atkin--Lehner involutions}
\label{subsec:AL}
Every cyclic isogeny can be factored into a composition of isogenies
of prime degree, by repeatedly applying the following well-known
fact.
\begin{prop}
\label{prop:factor}
  Let $E_1$ and~$E_2$ be elliptic curves defined over~$\Qbar$ and let
$\phi\colon E_1\to E_2$ be a cyclic isogeny of degree~$d$.  For any
factorization $d=d_1d_2$ into positive integers~$d_1$, $d_2$, there
exist an elliptic curve~$E$ and isogenies $\phi_1\colon E_1\to E$ and
$\phi_2\colon E\to E_2$ of degrees~$d_1$ and $d_2$ respectively such that
$\phi=\phi_2\circ\phi_1$.

$E$ is uniquely determined (up to isomorphism over~$\Qbar$) by the
ordered pair of factors~$(d_1,d_2)$, while $\phi_1$ and $\phi_2$ are
uniquely determined up to replacing~$(\phi_1,\phi_2)$
by~$(\alpha\circ\phi_1,\phi_2\circ\alpha^{-1})$ for
some~$\alpha\in\Aut(E)$.  In particular, if the curves do not have CM
then $\phi_1$ and $\phi_2$ are uniquely determined up to simultaneous
negation.
\end{prop}

\begin{proof}
  Since $\ker(\phi)$ is a cyclic subgroup of~$E_1(\Qbar)$ of order~$d$,
  it has a unique subgroup of order~$d_1$; this determines a
  cyclic $d_1$-isogeny~$\phi_1\colon E_1\to E$, unique up to
  post-composition with an automorphism of~$E$.  Since
  $\ker(\phi_1)\subseteq\ker(\phi)$ the original isogeny~$\phi$
  factors as~$\phi=\phi_2\circ\phi_1$, through an
  isogeny~$\phi_2\colon E\to E_2$ as required.
\end{proof}

In terms of $j$-invariants, such factorizations can be realized
without making additional field extensions, as follows.

\begin{cor}
\label{cor:same-field}
  Let $j_1,j_2\in\Qbar$ be isogenous and not CM, with
  $d=\deg(j_1,j_2)=d_1d_2$. Then there exists~$j\in\Qbar$
  with~$j_1\sim j\sim j_2$ such that $\deg(j_1,j)=d_1$
  and~$\deg(j,j_2)=d_2$.  Moreover, $j\in\Q(j_1,j_2)$.
\end{cor}
\begin{proof}
The first part is immediate from the proposition. For the last part,
let $K=\Q(j_1,j_2)$ and let $\phi\colon E_1\to E_2$ be a cyclic~$d$-isogeny
defined over~$K$, as in \Cref{cor:field-of-def}.  In the
notation of the proof of the proposition, $\ker(\phi)$ is defined
over~$K$, and so is $\ker(\phi_1)$ since it is the unique subgroup of
$\ker(\phi)$ of order~$d_1$.  Thus $E$ is also defined over~$K$, and
$j=j(E)\in K$.
\end{proof}

Again let $\phi\colon E_1\to E_2$ be a cyclic isogeny of degree~$d=d_1d_2$,
and now we assume that $\gcd(d_1,d_2)=1$.  Using both this
factorization and also $d=d_2d_1$, we obtain two elliptic curves~$E_3$
and~$E_4$ and a commutative diagram of isogenies as follows:
\[
\xymatrix@C+25pt { E_1 \ar[r]^{d_1} \ar[d]_{d_2} \ar[dr] & E_3 \ar[d]^{d_2}
  \ar[dl] \\ E_4 \ar[r]_{d_1} & E_2 }
\]
Here one diagonal is the original isogeny $E_1\to E_2$, while the
other diagonal gives a cyclic isogeny $E_3\to E_4$ of degree~$d$.  As
special cases, when $d_1=1$ this is the same as~$\phi$, while when
$d_2=1$, the new isogeny is the dual~$\hat\phi\colon E_2\to E_1$.

In terms of modular curves, $\phi\colon E_1\to E_2$ defines a non-cuspidal
point on~$X_0(d)$ and the new isogeny $E_3\to E_4$ is the image of
this point under the Atkin--Lehner involution~$W_{d_1}\colon X_0(d)\to
X_0(d)$.  As $d_1$ ranges over all positive divisors of~$d$ with
$\gcd(d_1,d/d_1)=1$, these involutions form an elementary abelian
$2$-group of order~$2^r$, where~$r$ is the number of distinct prime
factors of~$d$.  In particular, $W_1$ is the identity map while $W_d$
takes~$\phi$ to its dual.

Applying all such involutions, we obtain a collection
of $2^r$ isogenous curves, with isogeny degrees all such divisors
of~$d$.  We call this collection of curves and the isogenies between
them the \emph{Atkin--Lehner orbit} of the original isogeny
$\phi\colon E_1\to E_2$.

For example, the isogeny class of elliptic curves over~$\Q$ with LMFDB
label \lmfdbeciso{14}{a} consists of six curves labelled
{14.a1}-{14.a6}, linked by isogenies of degrees dividing~$18$.
\[
\xymatrix@C+25pt { 14.a1 \ar@{-}[r]^{3} \ar@{-}[d]_{2} & 14.a3 \ar@{-}[r]^{3} \ar@{-}[d]_{2}  & 14.a4 \ar@{-}[d]_{2}
  \\ 14.a2 \ar@{-}[r]_{3} & 14.a6 \ar@{-}[r]_{3} & 14.a5}
\]
There is an~$18$-isogeny from~{14.a1} to~{14.a5};
applying $W_9$ gives the $18$-isogeny from~{14.a4} to~{14.a2}, while
applying $W_2$ gives the dual $18$-isogeny from~{14.a2} to~{14.a4}. Thus
the Atkin--Lehner orbit consists of four of the six curves in the
class.  The other two curves, {14.a3} and {14.a6}, are $3$-isogenous to
these (and $2$-isogenous to each other) and may be obtained via the
construction in~\Cref{prop:factor}.   If we start with the $6$-isogeny
from~{14.a1} to~{14.a6}, then applying $W_2$ takes it to the $6$-isogeny
from~{14.a2} to~{14.a3}.

\subsubsection{The isogeny graph}  We may turn~$\Qbar$ into
a graph by introducing edges between isogenous pairs~$(j_1,j_2)$.
Each isogeny class is then a complete graph.  More useful is to only
include edges of \emph{prime} degree.  Since every cyclic isogeny is a
composite of isogenies of prime degree, this has the same connected
components.  The component whose vertex set is the isogeny class of
$j\in\Qbar$ is denoted $[j]$.

Fix a prime~$\ell$ and a non-CM $j\in\Qbar$.
We will construct graphs derived from the isogeny
class~$[j]$, representing only isogenies of $\ell$-power degree.  This
can be done in two ways, either as a subgraph of the isogeny graph or
as a quotient.  We briefly describe the subgraph construction, as used
by Elkies in~\cite{elkies}, before turning to the quotient
construction that we use in what follows.

\subsubsection{The $\ell$-primary subgraph.}
Consider the full subgraph of~$[j]$ consisting only of the
vertices~$j'$ such that $\deg(j,j')$ is a power of~$\ell$.  This is a
regular tree, each vertex having degree~$\ell+1$, often called a
Bruhat--Tits tree.  The graph~$[j]$ is the disjoint union of such
trees, where two vertices~$j',j''$ lie in the same subgraph if and
only if $\deg(j',j'')$ is a power of~$\ell$.  For each~$j$ in the
isogeny class there is a projection from~$[j]$ to the subgraph
containing~$j$, mapping each~$j'$ to the unique~$j''$ such that
$\deg(j,j'')$ is a power of~$\ell$ and $\deg(j'',j')$ is coprime
to~$\ell$.

\subsubsection{The $\ell$-primary quotient graph.}
Alternatively, we define a new graph, also a regular tree of
degree~$\ell+1$, which is a \emph{quotient} of~$[j]$, and does not
depend on a choice of representative~$j$ in its isogeny class.  We
denote this quotient by~$[j]_{\ell}$ and the projection
$[j]\to[j]_{\ell}$ by~$\pi_{\ell}$.

With~$\ell$ fixed, define $j_1\approx j_2$ to mean that $j_1\sim j_2$
with $\deg(j_1,j_2)$ \emph{coprime to~$\ell$}.  This is an equivalence
relation that refines the relation of isogeny.  Denote by
$\pi_{\ell}(j)$ the equivalence class of~$j\in\Qbar$ under the new
relation; thus the isogeny class~$[j]$ is the disjoint union of
classes~$\pi_{\ell}(j')$ for $j'\sim j$.

Let $[j]_{\ell} = \{\pi_{\ell}(j') \mid j'\in[j]\}$ be the quotient
of~$[j]$ by~${}\approx{}$.  Since Galois conjugation preserves isogeny
degrees, we have a well-defined induced action of~$G$ on~$[j]_{\ell}$,
such that $g(\pi_{\ell}(j)) = \pi_{\ell}(g(j))$.

For $j_1,j_2\in[j]$ let $\deg_{\ell}(j_1,j_2)$ be the $\ell$-primary
part of~$\deg(j_1,j_2)$, and set
$\deg_{\ell}(\pi_{\ell}(j_1),\pi_{\ell}(j_2)) = \deg_{\ell}(j_1,j_2)$.
This is well-defined, since if $j_1'\approx j_1$ and~$j_2'\approx j_2$
then $\deg_{\ell}(j_1',j_2') = \deg_{\ell}(j_1,j_2)$.  These degrees
between vertices of~$[j]_{\ell}$ are, by definition, powers of~$\ell$.

The set~$[j]_{\ell}$ inherits a graph structure from~$[j]$.
Explicitly, there is an edge between $\pi_{\ell}(j_1)$
and~$\pi_{\ell}(j_2)$ if and only if
$\deg_{\ell}(\pi_{\ell}(j_1),\pi_{\ell}(j_2)) = \ell$ (that is, if and
only if $\deg(j_1,j_2)$ has $\ell$-valuation equal to~$1$).  This
graph is a regular tree (every vertex has degree~$\ell+1$), and~$G$
acts on~$[j]_{\ell}$ through automorphisms of the tree.

The following result was stated in terms of $\ell$-primary subgraphs
by Elkies.  The version here will play an important role in the
construction of the core of an isogeny class of~$\Q$-curves.

\begin{prop}[Chinese Remainder Theorem for isogenies]
  \label{prop:CRT}
  Let $j\in\Qbar$ be non-CM.
  \begin{enumerate}
  \item Let $j'\in[j]$.  Then for almost all primes~$\ell$ we have
    $\pi_{\ell}(j')=\pi_{\ell}(j)$.
  \item Conversely, for each collection of $j_{\ell}\in[j]$, one for
    each prime~$\ell$, with $\pi_{\ell}(j_{\ell})=\pi_{\ell}(j)$ for
    almost all~$\ell$, there exists a unique $j'\in[j]$ such that
    $\pi_{\ell}(j')=\pi_{\ell}(j_{\ell})$ for all~$\ell$.
  \end{enumerate}
\end{prop}

\begin{proof} (1)  We have $\pi_{\ell}(j')=\pi_{\ell}(j)$ if
and only if $\ell\nmid\deg(j,j')$, which is true for almost
all~$\ell$.

(2) Let $\ell_i$ for $1\le i\le r$ be the primes for which
$\pi_{\ell}(j_{\ell})\not=\pi_{\ell}(j)$.  (If there are none, then
$j'=j$ meets the conditions.)  To ease notation set $j_i=j_{\ell_i}$
for $1\le i\le r$.  For each~$i$, factor the isogeny $j\to j_i$ as
$j\to j_i' \to j_i$ where $\deg(j,j_i') = \deg_{\ell}(j,j_i) =
\ell_i^{e_i}$ (say) is a power of~$\ell_i$, and $\deg(j_i',j_i)$ is
coprime to~$\ell_i$, so that $\pi_{\ell_i}(j_i')=\pi_{\ell_i}(j_i)$.

Let $E$ be an elliptic curve with $j(E)=j$, and for each~$i$ let $E_i$
be a curve with $j(E_i)=j_i'$, and $\phi_i\colon E\to E_i$ a cyclic isogeny
degree~$\ell_i^{e_i}$.  Define $\phi\colon E\to E'$ to be an isogeny whose
kernel is the sum of the~$\ker(\phi_i)$, so that $\phi$ is cyclic of
degree~$\prod_i\ell_i^{e_i}$.  The isomorphism class of~$E'$ depends
only on~$\ker(\phi)$, since any other isogeny with the same kernel is
obtained by composing~$\phi$ with an isomorphism; set $j'=j(E')$. Then
for each~$i$ we can factor $\phi=\psi_i\circ\phi_i$ with
$\deg(\psi_i)$ coprime to~$\ell_i$, so $j'$ satisfies
$\pi_{\ell_i}(j') = \pi_{\ell_i}(j_i') = \pi_{\ell_i}(j_i)$ for $1\le
i\le r$, while $\pi_{\ell}(j')=\pi_{\ell}(j)$ for
$\ell\not=\ell_1,\dots,\ell_r$ since these primes do not
divide~$\deg(j,j')$.

For the uniqueness, if also $\pi_{\ell}(j'')=\pi_{\ell}(j_{\ell})$
holds for
all~$\ell$, then for all~$\ell$ we have $\pi_{\ell}(j') =
\pi_{\ell}(j'')$, so $\ell\nmid\deg(j',j'')$; hence $j'=j''$.
\end{proof}

\subsection{\texorpdfstring{$\Q$}{\textbf{Q}}-curves and \texorpdfstring{$\Q$}{\textbf{Q}}-numbers}

A \emph{$\Q$-curve} is an elliptic curve~$E$ defined over $\Qbar$ such
that $E$ is isogenous (over $\Qbar$) to all its Galois conjugates.
Since this definition only depends on the $\Qbar$-isomorphism class
of~$E$ and $\Qbar$ is algebraically closed, it is a property of the
algebraic number $j(E)$, so we define $j\in\Qbar$ to be a
\emph{$\Q$-number} if any elliptic curve $E/\Qbar$ with $j(E)=j$ is a
$\Q$-curve.

As is well-known\footnote{If $j\in\Qbar$ is CM of degree~$h$,
  associated to the imaginary quadratic discriminant~$D$ with class
  number~$h$, then the elliptic curves whose $j$-invariants are the
  conjugates of~$j$ are $\C/\a$ as $\a$ runs over
  a set of proper $\OO$-ideal class representatives where $\OO$ is the
  order of discriminant~$D$; these are all isogenous to $\C/\OO$ (see
  Chapter~2 of~\cite{SilvermanBook2}).
}, all elliptic curves with CM are
$\Q$-curves, so all CM $j$-invariants are $\Q$-numbers.  We will be
less interested in these, and will restrict ourselves to non-CM
$\Q$-numbers; that is, $\Q$-numbers that are not CM $j$-invariants.

For $j\in\Qbar$ we denote by $\jbar$ the finite set of Galois
conjugates of~$j$. Then the condition for $j$ to be a $\Q$-number is
that $\jbar \subseteq [j]$.  In fact, if any element of an isogeny
class is a $\Q$-number then they all are, so that the class is a union
of Galois orbits.
\begin{prop}
\label{prop:Qclass}
  Let $j\in\Qbar$ be a $\Q$-number. Then every $j'\sim j$ is a
  $\Q$-number.
\end{prop}
\begin{proof}
  Suppose that $j'\sim j$ where $j$ is a $\Q$-number.  Applying $g\in
  G$ to an isogeny~$j\to j'$ shows that $g(j)\sim g(j')$. Since $j\sim
  g(j)$ by hypothesis, it follows that $g(j')\sim j'$ for all~$g\in G$.
\end{proof}

We call an isogeny class consisting of $\Q$-numbers a
\emph{$\Q$-class}.  By \Cref{prop:Qclass}, the Galois
action of~$G$ on~$\Qbar$ restricts to an action on each $\Q$-class.
The simplest $\Q$-classes are isogeny classes containing a
rational~$j$, which we call \emph{rational} $\Q$-classes.  Any other
$\Q$-class, and the $\Q$-numbers in it, are called \emph{strict}.
One of our goals is to find the simplest conjugacy class in a general
$\Q$-class.  In terms of elliptic curves, being isogenous to an
elliptic curve with rational $j$-invariant certainly implies being a
$\Q$-curve, by \Cref{prop:Qclass}, but we are interested in
studying \emph{strict} $\Q$-curves: $\Q$-curves that are not
isogenous to a curve with rational $j$-invariant.

For example, $j = -43136 \sqrt{2} + 60992$ is a strict $\Q$-number,
being the $j$-invariant of the elliptic curve with LMFDB label
\lmfdbecnf{2.2.8.1}{4096.1}{a}{1} which is $2$-isogenous to its Galois
conjugate~\lmfdbecnf{2.2.8.1}{4096.1}{a}{2}.  (It will follow from the
results of this Appendix that there can be no rational number
isogenous to this~$j$.) By contrast, $j = -36872164 \sqrt{2} +
52151080$ is a non-strict $\Q$-number: it is the $j$-invariant of
\lmfdbecnf{2.2.8.1}{128.1}{a}{2}, but its isogeny class consists of
four curves including \lmfdbecnf{2.2.8.1}{128.1}{a}{1} which has
$j=128$.

Let $j\in\Qbar$ be a $\Q$-number.  We define\footnote{The isogeny
  degree of $j$ in this sense should not be confused with the degree
  of $j$ as an algebraic number, i.e., the degree of the
  field extension~$\Q(j)/\Q$.}  the \emph{isogeny degree} of~$j$ to be
the least common multiple of the degrees $\deg(j,g(j))$ for~$g\in G$.
This is clearly the same for Galois conjugate $\Q$-numbers, so we may
also refer to the isogeny degree of a conjugacy class~$\jbar$ of
$\Q$-numbers.

\begin{defn}\label{def:deltaQ}
For each $\Q$-class~$\QQ$, define
\[
\delta_{\QQ}\colon G \to \Q^*/(\Q^*)^2
\]
by
\[
\delta_{\QQ}(g) = \deg(j,g(j)) \pmod{(\Q^*)^2},
\]
where~$j\in\QQ$ is arbitrary.
\end{defn}

\begin{lem}
  \label{lem:deg-indep}
For each $\Q$-class $\QQ$, $\delta_{\QQ}$ is a well-defined group
homomorphism with finite image.
\end{lem}
\begin{proof}
We first show that $\delta_{\QQ}$ is well-defined, not depending on
the choice of $j\in\QQ$.  Let $j_1,j_2\in\QQ$.  Since
$\deg(j_1,j_2)=\deg(g(j_1),g(j_2))$ by \Cref{lem:G-invariance}, it
follows from \Cref{cor:square-deg} that
$\deg(j_1,g(j_1))\equiv\deg(j_2,g(j_2))\pmod{(\Q^*)^2}$.

Next, for $g,h\in G$ we have
\[
\deg(j,gh(j)) \equiv \deg(j,g(j))\deg(g(j),gh(j)) =
\deg(j,g(j))\deg(j,h(j)) \pmod{(\Q^*)^2}
\]
by \Cref{cor:square-deg} and \Cref{lem:G-invariance}, so
$\delta_{\QQ}$ is a group homomorphism.  Finiteness is immediate since
the conjugacy class~$\jbar$ is finite.
\end{proof}

\begin{cor}
If $\QQ$ is a rational $\Q$-class then $\deg(j,g(j))$ is a square for
all~$j\in\QQ$ and~$g\in G$.
\end{cor}
The converse to this is also true, and will be proved later (see
\Cref{prop1:N=1} below).  Note that it does not follow
directly from \Cref{lem:deg-indep} that the square class of the
isogeny degree is isogeny-invariant, though this will turn out to be true.
The key result of Elkies is that every $\Q$-class contains some
``central'' $j$ whose isogeny degree is square-free, this degree being the
square-free part of the isogeny degree of every element of the class.

\begin{defn}
For a $\Q$-class~$\QQ$, let $L_{\QQ}$ be the fixed field of
$\ker(\delta_{\QQ})$.
\end{defn}
Since the image of~$\delta_{\QQ}$ is a finite subgroup
of~$\Q^*/(\Q^*)^2$, it is an elementary abelian $2$-group, so
$L_{\QQ}$ is a finite polyquadratic extension of~$\Q$.  Its Galois
group is isomorphic to~$(\Z/2\Z)^{\rho(\QQ)}$ for
some~$\rho(\QQ)\ge0$.

Looking only at the $\ell$-primary part of the isogeny degree, and its
exponent modulo~$2$, we obtain a character~$\delta_{\ell}\colon
G\to\{\pm1\}$ that is trivial for all but finitely many~$\ell$, and
otherwise cuts out a quadratic extension of~$\Q$.  Let $r=r(\QQ)\ge0$
be the number of primes~$\ell$ for which $\delta_{\ell}$ is
nontrivial, and let~$\ell_1,\dots,\ell_{r(\QQ)}$ be these primes.  For
each~$\ell_i$, and for every~$j\in\QQ$, exactly half the
isogeny degrees~$\deg(j,g(j))$ have even $\ell_i$-valuation and half odd;
while for other primes~$\ell$, these valuations are all even.

Define the \emph{level} $N=N(\QQ)$ of the $\Q$-class~$\QQ$ to be the
product~$\ell_1\dots\ell_r$.  By definition, the level is square-free;
it also divides the isogeny degree of every~$j\in\QQ$, since half of
the isogeny degrees~$\deg(j,g(j))$ have odd $\ell_i$-valuation and
hence in particular are divisible by~$\ell_i$, so their $\gcd$ is
divisible by~$\ell_i$.  We will see later that the level is actually
\emph{equal} to the square-free part of the isogeny degree of
every~$j\in\QQ$, by showing that the isogeny degree has odd
$\ell$-valuation if and only if $\ell\mid N$; at this point we only
know the ``only if'' implication.

Clearly $\ker(\delta_{\QQ})$ contains the intersection
$\cap_i\ker(\delta_{\ell_i})$, which has index~$2^r$ in~$G$, but in
general these subgroups of~$G$ are not equal. Hence we have $\rho\le
r$.  However when $\rho=0$ then $\delta_{\QQ}$ is trivial, so $r=0$
also.

If $\QQ$ is a rational $\Q$-class, then $\rho(\QQ) = r(\QQ) = 0$,
$N(\QQ)=1$, and $L_{\QQ}=\Q$.  Again, the converse to this is also
true, and will be proved below in \Cref{prop1:N=1}.

When the class~$\QQ$ is fixed we will simplify notation and write
$\rho=\rho(\QQ)$, $r=r(\QQ)$, $N=N(\QQ)$, etc.

\subsubsection{Central classes}
In a $\Q$-class~$\QQ$, an element~$j$ and its conjugacy class~$\jbar$
are called~\emph{central} if their isogeny degree is square-free;
equivalently, $j$ is central if~$\deg(j,g(j))$ is square-free for
all~$g\in G$.  In the next section we will see that every $\Q$-class
contains at least one central class.  Here we assume the existence of
such a class and draw several conclusions about it.

\begin{thm}
\label{thm:central-class-properties}
  Let $\QQ$ be a $\Q$-class. Then for all central classes~$C\subset\QQ$,
  \begin{enumerate}
    \item $C\subseteq \cap_{j\in\QQ}\Q(j)$;
    \item $\Q(C)=L_{\QQ}$; in particular, $\Q(C)$ depends only on~$\QQ$;
    \item the isogeny degree of~$C$ is~$N(\QQ)$, and the set of
      degrees between elements of~$C$ depends only on~$\QQ$.
  \end{enumerate}
\end{thm}

\begin{proof}
(1) Let $g\in G$, $j\in\QQ$ and~$j_1\in C$.  If $g(j)=j$, then
  $\deg(j,j_1)=\deg(g(j),g(j_1))=\deg(j,g(j_1))$.  Since
  $\deg(j_1,g(j_1))$ is square-free, $g(j_1)=j_1$ by
  \Cref{cor:squarefree-deg}.  Hence $\Q(j_1)\subseteq\Q(j)$.

(2) Let $j\in C$, and use $j$ to define the map $\delta_{\QQ}$
  (see~\Cref{def:deltaQ} and~\Cref{lem:deg-indep}).  Now
  $g\in\ker(\delta_{\QQ})$ if and only if~$\deg(j,g(j))$ is a square,
  which is if and only if $j=g(j)$ since~$\deg(j,g(j))$ is
  square-free. Hence the restriction of~$g$ to~$\Q(C)$ is the identity
  if and only if $g\in\ker(\delta_{\QQ})$, so $\Q(C)=L_{\QQ}$ (by
  definition of~$L_{\QQ}$).

(3) The fact that all central classes have the same isogeny degree
  follows from \Cref{lem:deg-indep}, for if $j$ and~$j'$ are both
  central, then the degrees $\deg(j,g(j))$ and~$\deg(j',g(j'))$ are
  square-free numbers that are equivalent modulo squares, and hence
  equal.

Let $N$ be the isogeny degree of a central class~$C=\jbar$.  Using~$j$
to define the characters~$\delta_{\ell}$, we see that $\delta_{\ell}$
is nontrivial precisely when~$\ell\mid N$, so $N=N(\QQ)$, independent
of the choice of~$C$.
\end{proof}

Note that we have not yet proved the existence of any central classes,
even when $\rho=0$.

\subsubsection{Definition of the core}
Let $\QQ$ be a $\Q$-class.  We define a \emph{core} of~$\QQ$ to be the
Atkin--Lehner orbit of a central class.  Thus the existence of a core
will follow from the existence of at least one central class~$C$
in~$\QQ$.  By \Cref{thm:central-class-properties} and
\Cref{cor:same-field}, all $j$-invariants in the core lie in the
polyquadratic field~$L_{\QQ}$.

From the previous subsection we see that a core consists of $2^r$
elements, such that for any~$j$ in the core the degrees of the
isogenies to the other core elements are \emph{all} the divisors of
the level~$N=N(\QQ)$.  The core is the union of $2^{r-\rho}$ central
classes.  The isogenies of degree~$N$ between elements of the core
define a collection of $2^r$ distinct points on $X_0(N)$, which are
closed under the actions of both the Galois group~$G$ and the
group~$W$ of Atkin--Lehner involutions, and hence determine a rational
point on the quotient~$X_0^*(N) = X_0(N)/W$.

The simplest examples of a strict $\Q$-class are those with $\rho=1$,
where $L_{\QQ}$ is a quadratic field.  Quadratic $\Q$-curves have been the
subject of much study.  In this case, a central class consists of a
pair of $L_{\QQ}/\Q$-conjugate $j$-invariants linked by a cyclic isogeny of
square-free degree~$N = \ell_1\ell_2\dots\ell_r$.  The core consists
of the complete Atkin--Lehner orbit of this isogeny, which has $2^r$
elements, in $2^{r-1}$ conjugate pairs, the invariants in each pair
being linked by a cyclic $N$-isogeny.

For example, let $j=(-30862080 \sqrt{13} - 111275008) \sqrt{2}
- 43645440 \sqrt{13} - 157366464$, which is the $j$-invariant of the
$\Q$-curve~$E$ with LMFDB label
\lmfdbecnf{4.4.10816.1}{1.1}{a}{1}. Its isogeny class
\lmfdbecnfiso{4.4.10816.1}{1.1}{a} (over the biquadratic
field~$K=\Q(\sqrt{2},\sqrt{13})$) consists of the four Galois
conjugates of~$E$, linked by isogenies of all degrees dividing the
level, which is~$15$.  Here we have $\rho=r=2$ and $L_{\QQ}=K$.

\subsection{Construction of the core}

We now prove that central classes exist for every $\Q$-class.  Our
proof is similar to that of Elkies in~\cite{elkies}, except that we
use the quotient trees~$\QQ_{\ell}$ instead of subtrees, and we have
already established several useful preliminaries.  Let~$\QQ$ be a
$\Q$-class.  Applying the quotient construction, we obtain a
tree~$\QQ_\ell$ for every prime~$\ell$.  We now associate a finite
subtree $T_{\ell}(j) \subset \QQ_\ell$ to every conjugacy
class~$\jbar$ in~$\QQ$.  Let $N_0$ be the isogeny degree of~$j$.  The
image of~$\jbar$ under~$\pi_{\ell}$ is a finite subset
of~$\QQ_{\ell}$, which is a singleton unless $\ell$ divides~$N_0$.  We
define $T_{\ell}(j)$ to be the finite subtree of~$\QQ_{\ell}$ spanned
by the conjugate vertices~$\pi_{\ell}(g(j))\in\QQ_{\ell}$ for $g\in
G$.  Denote by~$n(\ell,j)$ the diameter of $T_{\ell}(j)$; by definition of
isogeny degree, this is the $\ell$-valuation of~$N_0$, being the
maximum $\ell$-valuation of~$\deg(j,g(j))$ for $g\in G$.  The leaves
(vertices of valency~$1$) of $T_{\ell}(j)$ are precisely the
vertices~$\pi_{\ell}(g(j))\in\QQ_{\ell}$, by construction and the
transitivity of the action of $G$ on~$\jbar$.

We will use a standard fact about finite trees~$T$
(see~\cite{Jordan}), that they have a unique \emph{centre}, that is
either a vertex (when the diameter of~$T$ is even) or an edge (when
the diameter is odd), such that every maximal path in the tree passes
through the centre.  Since automorphisms of~$T$ take maximal paths to
maximal paths, the centre of~$T$ is fixed by all automorphisms.
Recall also that in a tree there is a unique path between any two
vertices, whose length defines the \emph{distance} between the
vertices.  In $T_{\ell}(j)$ the distance between two leaves
$\pi_{\ell}(j), \pi_{\ell}(j')$ is~$d$ where $\deg_{\ell}(j, j') =
\ell^d$.

In our situation, $T_{\ell}(j)$ has a central vertex or edge when its
diameter~$n(\ell,j)$ is even or odd (respectively), and this centre is
fixed by the action of~$G$.

\begin{prop}
  The following are equivalent:
  \begin{enumerate}
  \item $\ell\nmid N(\QQ)$;
  \item $n(\ell,j)$ is even;
  \item the distance between any two leaves of $T_{\ell}(j)$ is even;
  \item $G$ has at least one fixed point in $T_{\ell}(j)$;
  \item $G$ has at least one fixed point in $\QQ_\ell$;
  \end{enumerate}
\end{prop}

\begin{proof}
  $(1)\Leftrightarrow(3)$: by definition of the level.

  $(3)\Rightarrow(2)$: obvious.

  $(2)\Rightarrow(3), (4)$: When $n(\ell,j)=2m$ is even, $T_{\ell}(j)$
  has a central point~$\pi_{\ell}(j_0)$, fixed by~$G$.  Every leaf is
  at distance~$m$ from the centre, so the distances between leaves are
  all even.

  $(4)\Rightarrow(5)$: obvious.

  $(5)\Rightarrow(2)$ (by contrapositive): Suppose that $n(\ell,j)=2m+1$ is
  odd, and that the central edge of~$T_{\ell}(j)$
  is~$\pi_{\ell}(j_1)-\pi_{\ell}(j_2)$.  Every $g\in G$ either fixes
  both vertices~$\pi_{\ell}(j_1),\pi_{\ell}(j_2)$, or it interchanges
  them.  Every leaf is at distance~$m$ from one of the central
  vertices and distance $m+1$ from the other.  Since $G$ acts
  transitively on the leaves and preserves distance, there
  exists~$g_0\in G$ that does interchange $\pi_{\ell}(j_1)$
  and~$\pi_{\ell}(j_2)$.  Moreover, such elements~$g_0$ have no fixed
  points at all in~$\QQ_{\ell}$, since for all such points their
  distances from these two central vertices differ by~$1$; in fact,
  such~$g_0$ interchange the subset of vertices of~$\QQ_{\ell}$ that
  are closer to $\pi_{\ell}(j_1)$ than to~$\pi_{\ell}(j_2)$ with its
  complement.
\end{proof}

\begin{cor}
With the same notation, when $n(\ell,j)=2m+1$ is odd,  the distance between
any two  leaves of $T_{\ell}(j)$ is either even and at most~$2m$, or
is equal to the diameter~$2m+1$.
\end{cor}
\begin{proof}
If two leaves are on the same side of the central edge then they are
both at distance~$m$ from the closest central vertex, and hence the
distance between them is even and at most~$2m$.  If they are on
different sides, then the path between them passes through the central
edge and has length~$2m+1$.
\end{proof}

\begin{cor}
For all~$j\in\QQ$, the square-free part of the isogeny degree of~$j$ is equal
to the level~$N(\QQ)$.
\end{cor}
\begin{proof}
This is $(1)\Leftrightarrow(2)$ of the proposition, since $n(\ell,j)$ is the
$\ell$-valuation of the isogeny degree of~$j$.
\end{proof}

Hence we have another characterization of the level of a
$\Q$-class~$\QQ$: it is the product of the (finitely many)
primes~$\ell$ such that $G$ has no fixed points on~$\QQ_{\ell}$, or
equivalently the primes~$\ell$ such that for all~$j\in\QQ$ there
exists~$g\in G$ such that $\ord_{\ell}\deg(j,g(j))$ is odd.

\begin{thm}
\label{thm:core-exists}
Every $\Q$-class has at least one central conjugacy class and hence a
core.
\end{thm}

Before handling the general case, we start with the case of a
$\QQ$-class of level~$1$; that is, such that every~$j$ has square
isogeny degree.

\begin{prop}
  \label{prop1:N=1}
  Let $\QQ$ be a $\Q$-class.  Then $N(\QQ)=1$ if and only if~$\QQ$ is
  rational.
\end{prop}
\begin{proof}
If $\QQ$ contains a rational~$j$ then $r(\QQ)=0$ and $N(\QQ)=1$.  For
the converse, suppose that $N(\QQ)=1$; we must show that there exists
a rational $j_0\in\QQ$.

Let $j\in\QQ$ be arbitrary.  Since $N=1$, for every prime~$\ell$ the
tree~$T_{\ell}(j)$ has a central vertex $\pi_{\ell}(j_{\ell})$.  For
all but finitely many~$\ell$, $T_{\ell}(j)$ is a singleton and we may
take $j_{\ell}=j$; in all cases we may choose~$j_{\ell}$ (in its class
with respect to~${}\approx{}$) with $\deg(j,j_{\ell})$ a power
of~$\ell$.  By \Cref{prop:CRT}, there exists a (unique) $j_0\in\QQ$
such that $\pi_{\ell}(j_0) = \pi_{\ell}(j_{\ell})$ for all~$\ell$. We
claim that $j_0\in\Q$.

Let $g\in G$ and let~$\ell$ be any prime.  Since $\pi_{\ell}(j_0) =
\pi_{\ell}(j_{\ell})$, it follows that $\deg(j_0, j_{\ell})$ and
$\deg(g(j_0), g(j_{\ell}))$ are prime to~$\ell$.  Since $g$
fixes~$\pi_{\ell}(j_{\ell})$, also $\deg(j_{\ell},g(j_{\ell}))$ is
prime to~$\ell$.  Hence $\deg(j_0,g(j_0))$ is prime to~$\ell$.  As
this holds for all primes, $\deg(j_0,g(j_0))=1$; that is,
$g(j_0)=j_0$. This holds for all~$g\in G$, so $j_0\in \Q$.
\end{proof}

In this minimal case, the core of the class is a singleton consisting
of a single rational $j$-invariant.  In general the core is not
unique, as any rational number in the class is a core.  However, the
number of these is finite.  To see this, we may combine the fact that
over $\Q$ all isogeny classes of elliptic curves are finite with
\Cref{lem:twist}.

\begin{proof}[Proof of \Cref{thm:core-exists}]
Choose some element~$j$ in the $\Q$-class~$\QQ$. Let
$\ell_1,\dots,\ell_r$ be the prime factors of~$N(\QQ)$, and for
each~$i$ let the two vertices of the central edge of $T_{\ell_i}(j)$
be $\pi_{\ell}(j_i^\pm)$.  Now, for each choice of signs
$s=(s_1,\dots,s_r)\in\{\pm\}^r$, the Chinese Remainder construction
yields $j_s\in\QQ$ such that $\pi_{\ell_i}(j_s) =
\pi_{\ell_i}(j_i^{s_i})$ for all~$i$, and
also~$\pi_{\ell}(j_s)=\pi_{\ell}(j)$ for~$\ell\nmid N(\QQ)$.  Let
$C=\{j_s\mid s\in\{\pm\}^r\}$.  Since $G$ permutes each
pair~$\pi_{\ell_i}(j_i^{\pm})$, it follows (as in the proof of
\Cref{prop1:N=1}) that $G$ acts on~$C$.  Explicitly, each $g\in G$
either fixes both $\pi_{\ell_i}(j_i^{\pm})$ or swaps them over, so
maps $j_s$ to~$j_{s'}$ where $s'$ is obtained from~$s$ by changing
some of the signs.  Moreover, $\deg(j_s,g(j_s)) =
\prod_i\ell^{\eps_i}$ where $\eps_i=0$ when $g(\pi_{\ell_i}(j_i^{+}))
= \pi_{\ell_i}(j_i^{+})$ and $\eps_i=1$ when $g(\pi_{\ell_i}(j_i^{+}))
= \pi_{\ell_i}(j_i^{-})$; in particular, $\deg(j_s,g(j_s))$ is
square-free.  Hence the conjugacy class $G(j_s)$ is central.

Thus the action of~$G$ on~$C$ factors faithfully through a
homomorphism~$G\to(\Z/2\Z)^r$, namely $g\mapsto
(\eps_1,\dots,\eps_r)$. The image has order~$2^{\rho(\QQ)}$, and $C$
is the closure of $G(j_s)$ under Atkin--Lehner involutions.
\end{proof}

The following corollary is useful for the $\Q$-curve testing
algorithm, since it shows that the isogenies from a $\Q$-curve~$E$
defined over a number field $K$ to the central curves in its isogeny
class are defined over $K$ itself.  It also determines the smallest
degree of an isogeny from a $\Q$-curve $E$ to a central $\Q$-curve in
terms of the isogeny degree of~$E$ (that is, the least common multiple
of the degrees of the isogenies between $E$ and its Galois
conjugates).

\begin{cor}
\label{cor:isog-degree}
Let $K$ be a number field and let $E$ be a non-CM $\Q$-curve defined
over~$K$.
\begin{enumerate}
\item\label{pt1}
There exists a central $\Q$-curve~$E_0$ with an isogeny~$\phi\colon E\to
E_0$, where both~$E_0$ and~$\phi$ are also defined over~$K$.
\item\label{pt2} Let $N$ be the isogeny degree of~$E$, and write
  $N=N_0M^2$ with $N_0$ square-free.  Then the smallest degree of an
  isogeny~$\phi\colon E\to E_0$ (as in part ~\ref{pt1}) is~$M$, and
  the isogenies from~$E$ to the conjugates of~$E_0$ have degree~$Mn$
  for~$n\mid N_0$.
\end{enumerate}
\end{cor}
\begin{proof}
(\ref{pt1}) Let $j=j(E)\in K$. By~\Cref{thm:central-class-properties}, the
  core $j$-invariants are all in~$K$.  Taking~$j_0$ to be any of
  these, we may take~$E_0$ to be an elliptic curve defined over the
  core field~$\Q(j_0)$ (and hence also defined over~$K$), choosing the
  quadratic twist so that the isogeny~$\phi\colon E\to E_0$ is also defined
  over $\Q(j,j_0)\subseteq K$ as in~\Cref{lem:twist}.

(\ref{pt2}) Write the isogeny degree of~$j$ as $N_0M^2$ as in the
  statement.  We use the notation of the proof of
  \Cref{thm:core-exists}, so $N_0=\ell_1\dots\ell_r$.  The degrees of
  the isogenies between~$j$ and the $2^r$ central $j$-invariants have
  $\ell_i$-valuation either $m_i$ or~$m_i+1$, where $m_i$ is the
  valuation of~$M$ and $2m_i+1$ that of~$N$.  Taking the product over
  all~$\ell_i$ gives the result stated.
\end{proof}

\end{document}